\documentclass[a4paper,11pt]{amsart}
\usepackage[titletoc,toc,title]{appendix}
\usepackage[left=2.7cm,right=2.7cm,top=3.5cm,bottom=3cm]{geometry}
\usepackage{amssymb,latexsym,amsmath,amsthm,amscd}
\usepackage{graphicx}
\usepackage{dsfont}
\usepackage{marginnote}
\usepackage[all]{xy}
\usepackage{mathrsfs}
\usepackage{color}
\definecolor{Blue}{rgb}{0.3,0.3,0.9}

\usepackage[T2A,OT1]{fontenc}
\DeclareSymbolFont{cyrillic}{T2A}{cmr}{m}{n}
\DeclareMathSymbol{\Sha}{\mathalpha}{cyrillic}{216}

\newcommand{\sk}{\vspace{0.1in}}
\newtheorem{thm}{Theorem}[section]
\newtheorem{def-thm}[thm]{Definition-Theorem}
\newtheorem{cor}[thm]{Corollary}
\newtheorem{lem}[thm]{Lemma}
\newtheorem{def-lem}[thm]{Definition-Lemma}
\newtheorem{prop}[thm]{Proposition}
\newtheorem{conj}[thm]{Conjecture}
\newtheorem*{ThmA}{Theorem A}
\newtheorem*{ThmB}{Theorem B}
\newtheorem*{ThmC}{Theorem C}
\newtheorem*{ThmD}{Theorem D}

\theoremstyle{definition}
\newtheorem{defn}[thm]{Definition}
\theoremstyle{remark}
\newtheorem{rem}[thm]{Remark}
\numberwithin{thm}{section}
\numberwithin{equation}{section}

\newcommand{\cO}{\mathcal{O}}

\newcommand{\pp}{\mathfrak{p}}

\newcommand{\qq}{\mathfrak{q}}

\newcommand{\bQ}{\mathbf{Q}}
\newcommand{\bR}{\mathbf{R}}
\newcommand{\bZ}{\mathbf{Z}}
\newcommand{\bC}{\mathbf{C}}

\def\ac{{\rm ac}}

\newcommand{\Tc}{{\mathbf{T}^{\rm ac}}}

\newcommand{\unr}{R_0}

\newcommand{\hooklongrightarrow}{\lhook\joinrel\longrightarrow}

\begin{document}

\title[On the Iwasawa main conjectures for modular forms]{On the Iwasawa main conjectures for modular forms at non-ordinary primes}

\author[F.~Castella]{Francesc Castella}
\address[Castella]{Department of Mathematics, Princeton University, Princeton, NJ 08544-1000, USA}
\email{fcabello@math.princeton.edu}
\author[M.~\c{C}iperiani]{Mirela \c{C}iperiani}
\address[\c{C}iperiani]{Department of Mathematics, The University of Texas at Austin, Texas 78712, USA}
\email{mirela@math.utexas.edu}
\author[C.~Skinner]{Christopher Skinner}
\address[Skinner]{Department of Mathematics, Princeton University, Princeton, NJ 08544-1000, USA}
\email{cmcls@math.princeton.edu}
\author[F.~Sprung]{Florian Sprung}
\address[Sprung]{Department of Mathematics, Arizona State University, Tempe, AZ 85287-1804, USA}
\email{florian.sprung@asu.edu}

\thanks{This material is based upon work supported by the National Science Foundation under grant agreements DMS-1801385 (F.C.), DMS-1128155 and  DMS-1352598 (M.C.), and DMS-1301842 and DMS-1501064 and by the Simons Investigator Grant \#376203 (C.S.).  
This project has received funding from 
the European Research Council under the European Union's Horizon 2020 research and innovation programme under grant agreement \#682152 (F.S.).}


\date{\today}



\begin{abstract}
In this paper, we prove under mild hypotheses 
the Iwasawa main conjectures of Lei--Loeffler--Zerbes 
for modular forms of weight $2$ at non-ordinary primes.   
Our proof is based on the study of the two-variable analogues of these conjectures formulated by B\"uy\"ukboduk--Lei 
for imaginary quadratic fields in which $p$ splits, and on anticyclotomic Iwasawa theory. 
As application of our results, we deduce the $p$-part of the Birch and Swinnerton-Dyer formula in analytic ranks $0$ or $1$ for abelian varieties over $\bQ$ of ${\rm GL}_2$-type for non-ordinary primes $p>2$.	
\end{abstract}

\maketitle

\setcounter{tocdepth}{2}
\tableofcontents

\section{Introduction}

\subsection{Main results}

Let $f=\sum_{n=1}^\infty a_nq^n\in S_2^{\rm new}(\Gamma_0(N))$ be a 
newform of weight $2$ and trivial nebentypus. Let $p>2$ be a prime and let $L$ be a finite extension of $\bQ_p$ with ring of integers $\cO_L$ containing the image of the Fourier coefficients $a_n$ under a fixed embedding $\imath_p:\overline{\bQ}\hookrightarrow\overline{\bQ}_p$. Assume that $p$ is a good prime for $f$, i.e., $p\nmid N$,  and non-ordinary in the sense that $\vert\imath_p(a_p)\vert_p<1$. Let 
\[
\rho_f:G_\bQ:={\rm Gal}(\overline{\bQ}/\bQ)\longrightarrow{\rm Aut}_L(V)\simeq{\rm GL}_2(L)
\]
be the Galois representation associated with $f$ by Deligne, and fix a $G_\bQ$-stable $\cO_L$-lattice $T^*$ in the contragredient of $V$. Let $\Gamma^{\rm cyc}={\rm Gal}(\bQ_\infty/\bQ)$ be the Galois group of the cyclotomic $\bZ_p$-extension of $\bQ$, and let $\mathbf{T}^{\rm cyc}$ be the cyclotomic deformation of $T^*$, defined as the module 
\begin{equation}\label{def:T-cyc}
\mathbf{T}^{\rm cyc}:=T^*\otimes_{\cO_L}\cO_L[[\Gamma^{\rm cyc}]]
\end{equation}
equipped with the Galois action given by $\rho_f^*\otimes\Psi$, where $\Psi:G_\bQ\twoheadrightarrow\Gamma^{\rm cyc}\hookrightarrow\cO_L[[\Gamma^{\rm cyc}]]^\times$ is the universal cyclotomic character. Also, let $\cO_L[[\Gamma^{\rm cyc}]]^\vee$ 
be the Pontrjagin dual of $\cO_L[[\Gamma^{\rm cyc}]]$, and set
\begin{equation}\label{def:A-cyc}
\mathbf{A}^{\rm cyc}:=T^*\otimes_{\cO_L}\cO_L[[\Gamma^{\rm cyc}]]^\vee
\end{equation}
equipped with the Galois action given by $\rho_f^*\otimes\Psi^{-1}$.  
\sk

Extending earlier of Kobayashi \cite{kobayashi-152} and Sprung \cite{sprung-JNT}, 
Lei--Loeffler--Zerbes constructed in \cite{LLZ-AJM} signed Coleman maps
\[
{\rm Col}^\sharp,\;{\rm Col}^\flat:H^1(\bQ_p,\mathbf{T}^{\rm cyc})\longrightarrow\cO_L[[\Gamma^{\rm cyc}]].
\]
These were used to formulate variants of the Iwasawa main conjecture of Mazur--Swinnerton-Dyer \cite{mazur-swd} in the supersingular setting, relating the characteristic ideal of the Pontrjagin dual of certain `signed' Selmer groups ${\rm Sel}^\bullet(\bQ,\mathbf{A}^{\rm cyc})\subset H^1(\bQ,\mathbf{A}^{\rm cyc})$ to corresponding `signed' $p$-adic $L$-functions 
\begin{equation}\label{eq:Lp-Q}
L_p^\bullet(f/\bQ):={\rm Col}^\bullet({\rm res}_p(\mathbf{z}^{\rm Kato})),\nonumber
\end{equation}
where $\mathbf{z}^{\rm Kato}\in H^1(\bQ,\mathbf{T}^{\rm cyc})$ is constructed from Kato's Euler system. 

\begin{conj}[Lei--Loeffler--Zerbes {\cite{LLZ-AJM}}]\label{conj:LLZ}
For each $\bullet\in\{\sharp,\flat\}$ the module ${\rm Sel}^\bullet(\bQ,\mathbf{A}^{\rm cyc})$ is $\cO_L[[\Gamma^{\rm cyc}]]$-cotorsion, and we have
\[
{\rm Char}_{\cO_L[[\Gamma^{\rm cyc}]]}({\rm Sel}^{\bullet}(\bQ,\mathbf{A}^{\rm cyc})^\vee)=(L_p^{\bullet}(f/\bQ))
\]
as ideals in $\cO_L[[\Gamma^{\rm cyc}]]$.
\end{conj}


As shown in \cite[$\S{6}$]{LLZ-AJM}, if $\bullet\in\{\sharp,\flat\}$ is such that $L_p^\bullet(f/\bQ)$ is nonzero, 
then Conjecture~\ref{conj:LLZ} is \emph{equivalent} to Kato's  main conjecture \cite[Conj.~12.10]{Kato295}; in particular, one of the divisibilities predicted by Conjecture~\ref{conj:LLZ} follows from Kato's work \cite[Thm.~12.5]{Kato295}. As one of our main results, in this paper we prove the opposite divisibility, thus yielding the following result:

\begin{ThmA}\label{cor:cyc-IMC}
Assume that the level of $f\in S_2^{\rm new}(\Gamma_0(N))$ is square-free.  
If $\bullet\in\{\sharp,\flat\}$ is such that $L_p^\bullet(f/\bQ)$ is nonzero, then ${\rm Sel}^{\bullet}(\bQ,\mathbf{A}^{\rm cyc})$ is $\cO_L[[\Gamma^{\rm cyc}]]$-cotorsion, and we have
	\[
	{\rm Char}_{\cO_L[[\Gamma^{\rm cyc}]]}({\rm Sel}^{\bullet}(\bQ,\mathbf{A}^{\rm cyc})^\vee)=(L_p^{\bullet}(f/\bQ))
	\]
	as ideals in $\cO_L[[\Gamma^{\rm cyc}]]$. In particular, Kato's main conjecture for $f$ holds.
\end{ThmA}

\begin{rem}
\begin{enumerate}
\item {} It follows from Rohrlich's work \cite{rohrlich-Q} that at least one of the $p$-adic $L$-functions $L_p^\bullet(f/\mathbf{Q})$ is nonzero, and one expects that both of them are. Thus Theorem~A should suffice to establish the two cases of Conjecture~\ref{conj:LLZ}.
\item {} For a newform $f$ with \emph{rational} Fourier coefficients, hence corresponding to an isogeny class of elliptic curves $E/\bQ$ by the Eichler--Shimura construction, the main conjectures of \cite{LLZ-AJM} reduce to the signed main conjectures of Kobayashi \cite{kobayashi-152} (case $a_p=0$) 
and give a formulation that is similar to that of Sprung \cite{sprung-JNT} (general supersingular case); see \cite[\S{5}]{LLZ-ANT} for a detailed direct comparison. In any case, all the signed main conjecture are equivalent, since they are equivalent to Kato's \cite[Conj.~12.10]{Kato295}, and hence in the elliptic curve case Theorem~A was first established by Xin~Wan \cite{wan-combined} when $a_p=0$ and by Sprung \cite{sprung-IMC} when $p\mid a_p$ but $a_p\neq 0$. 
\end{enumerate}
\end{rem}

In a recent work \cite{BL-non-ord}, B{\"u}y{\"u}kboduk--Lei have formulated analogues of the main conjectures of \cite{LLZ-AJM} for non-ordinary newforms $f$ based-changed to an imaginary quadratic field, and our proof of Theorem~A relies on the progress we achieve in this paper towards the main  conjectures in \cite{BL-non-ord}. To state our results in this direction, let $K$ be an imaginary quadratic field in which
\[
\textrm{$p\cO_K=\pp\overline{\pp}$ splits.} 
\]
Let $\Gamma_K={\rm Gal}(K_\infty/K)$ be the Galois group of the  $\bZ_p^2$-extension of $K$, and define the 
deformations $\mathbf{T}$ and $\mathbf{A}$ of $T^*$ and its Pontrjagin dual similarly as in (\ref{def:T-cyc}) and (\ref{def:A-cyc}), with $\Gamma_K$ in place of $\Gamma^{\rm cyc}$. 
For each $\qq\in\{\pp,\overline{\pp}\}$ there are signed Coleman maps
\begin{equation}\label{eq:intro-Col}
{\rm Col}_\qq^\sharp,\;{\rm Col}_\qq^\flat:H^1(K_\qq,\mathbf{T})\longrightarrow
\cO_L[[\Gamma_K]]
\end{equation}
which can be used to define the local conditions at $p$ cutting out certain `doubly-signed' Selmer groups  $\mathfrak{Sel}^{\bullet,\circ}(K,\mathbf{A})\subset H^1(K,\mathbf{A})$, 
with $\bullet, \circ\in\{\sharp,\flat\}$. On the other hand, developing the ideas in the elliptic curve case  \cite{wan-combined, sprung-IMC}, B{\"u}y{\"u}kboduk--Lei \cite{BL-non-ord} construct certain `signed' classes
\begin{equation}\label{eq:intro-BF}
\mathcal{BF}^\sharp,\;\mathcal{BF}^\flat\in H^1(K,\mathbf{T})
\end{equation}
obtained from a suitable decomposition of the (unbounded) Beilinson--Flach classes of Loeffler--Zerbes \cite{LZ-Coleman}, just as the
signed $p$-adic $L$-functions \cite{pollack, sprung-ANT} decompose the unbounded $p$-adic $L$-functions of Amice--Velu and Vishik. Letting
\begin{equation}\label{eq:intro-L}
\mathfrak{L}_p^{\bullet,\circ}(f/K):={\rm Col}_{\overline{\pp}}^\circ({\rm res}_{\overline\pp}(\mathcal{BF}^\bullet)),
\end{equation}
where ${\rm res}_{\overline\pp}:H^1(K,\mathbf{T})\rightarrow H^1(K_{\overline\pp},\mathbf{T})$ is the restriction map,  the `doubly-signed' main conjectures of \cite{BL-non-ord} (see  [\emph{loc.cit.}, Conj.~4.15]) then predict that the Pontrjagin dual of 
$\mathfrak{Sel}^{\bullet,\circ}(K,\mathbf{A})$ is $\cO_L[[\Gamma_K]]$-cotorsion, with characteristic ideals generated by $(\ref{eq:intro-L})$: 

\begin{conj}[{B{\"u}y{\"u}kboduk--Lei} {\cite{BL-non-ord}}]\label{conj:BL-IMC}
For each $\bullet, \circ\in\{\sharp,\flat\}$ the module $\mathfrak{Sel}^{\bullet,\circ}(K,\mathbf{A})$ is $\cO_L[[\Gamma_K]]$-cotorsion,  and
\[
{\rm Char}_{\cO_L[[\Gamma_K]]}(\mathfrak{Sel}^{\bullet,\circ}(K,\mathbf{A})^\vee)
=(\mathfrak{L}_p^{\bullet,\circ}(f/K))
\]
as ideals in $\cO_L[[\Gamma_K]]$.	
\end{conj}

\begin{rem} 
More generally, the doubly-signed main conjectures of \cite{BL-non-ord} are formulated for possibly higher even weights $k<p$, but in this paper we shall restrict to the weight $2$ case.
\end{rem}

In order to state our main result towards Conjecture~\ref{conj:BL-IMC}, note that there is a 
decomposition
\[
\Gamma_K\simeq\Gamma_K^{+}\times\Gamma_K^{-} 
\]
induced by the action of the non-trivial automorphism of $K$, in which  $\Gamma_K^{+}$ is identified with the Galois group $\Gamma^{\rm cyc}$ of the cyclotomic $\bZ_p$-extension of $K$. 
 
\begin{ThmB}\label{thm:main}
Assume that:
	\begin{itemize}
		\item[(i)]{} $N$ is square-free;
		\item[(ii)]{} $\bar{\rho}_f$ is ramified at every prime $\ell\mid N$ which is non-split in $K$, and there is at least one such prime,
		\item[(iii)]{} $\bar{\rho}_f\vert_{{\rm Gal}(\overline{\bQ}/K)}$ is irreducible.
	\end{itemize}
	If $\bullet\in\{\sharp,\flat\}$ is such that $\mathfrak{L}_p^{\bullet,\bullet}(f/K)$ is non-zero at some character factoring through $\Gamma_K\twoheadrightarrow\Gamma^{\rm cyc}$, then the module 
	$\mathfrak{Sel}^{\bullet,\bullet}(K,\mathbf{A})$ is $\cO_L[[\Gamma_K]]$-cotorsion, and we have
	\[
	{\rm Char}_{\cO_L[[\Gamma_K]]}(\mathfrak{Sel}^{\bullet,\bullet}(K,\mathbf{A})^\vee)
	=(\mathfrak{L}_p^{\bullet,\bullet}(f/K))
	\]
	as ideals in $\cO_L[[\Gamma_K]]$. 
\end{ThmB}


\begin{rem}
The signed Iwasawa main conjectures we study in this paper (i.e., the $1$- and $2$-variable signed main conjectures of \cite{LLZ-AJM} and \cite{BL-non-ord})  are formulated using Berger's theory of Wach modules \cite{berger-explicit}, \cite{berger-limit}. For an alternative  formulation of signed main conjectures 
using Fontaine's theory of group schemes, and corresponding analogues of our Theorems~A and B, see the forthcoming work \cite{KS}.
\end{rem} 

%

\subsection{Outline of the proofs} 

Similarly as in \cite{SU}, the proof of Theorem~A will be deduced from an application of Theorem~B for a suitable choice of auxiliary imaginary quadratic field $K$, together with Kato's work. 
The field $K$ determines a factorization
\[
N=N^+ N^-
\]
with $N^+$ (resp. $N^-$) divisible only by primes which are either split or ramified (resp. inert) in $K$. Let $\nu(N^-)$ denote the number of prime factors of the square-free integer $N$. The proof of Theorem~B is significantly different depending upon whether:
\begin{itemize} 
\item{} $\nu(N^-)$ is odd: the \emph{definite} case, or 
\item{} $\nu(N^-)$ is even: the \emph{indefinite} case.
\end{itemize}
We first outline the  proof in the latter case. Let 
\[
\Gamma^{\rm ac}:=\Gamma_K^-={\rm Gal}(K_\infty^{\rm ac}/K)
\] 
be the Galois group of the anticyclotomic $\bZ_p$-extension of $K$, and let $\mathbf{T}^\ac$ and $\mathbf{A}^{\ac}$ be the corresponding deformations of $T^*$ and its Pontrjagin dual, respectively. Similarly as in the construction 
in \cite{BL-non-ord} of the bounded classes $(\ref{eq:intro-BF})$, 
we construct `signed' Heegner classes
\[
\mathfrak{Z}^\sharp,\;\mathfrak{Z}^\flat\in H^1(K,\mathbf{T}^{\rm ac})
\]
decomposing the Kummer images of classical Heegner points over 
$K_\infty^{\rm ac}/K$.
For each $\bullet\in\{\sharp,\flat\}$, we show that   $\mathfrak{Z}^\bullet$ lands in the signed Selmer group $\mathfrak{Sel}^{\bullet,\bullet}(K,\mathbf{T}^{\rm ac})$ obtained by restricting the two-variable $\mathfrak{Sel}^{\bullet,\bullet}(K,\mathbf{T})$ to the `anticyclotomic line'. Extending the methods introduced in \cite{cas-wan-ss} in the  
case $a_p=0$, we then formulate and prove (under mild hypotheses) a natural extension of Perrin-Riou's main conjecture \cite{PR-HP}, predicting that both $\mathfrak{Sel}^{\bullet,\bullet}(K,\mathbf{T}^{\rm ac})$ and the Pontrjagin dual $\mathfrak{Sel}^{\bullet,\bullet}(K,\mathbf{A}^{\rm ac})^\vee$ have  $\cO_L[[\Gamma^{\rm ac}]]$-rank one, with
\begin{equation}\label{eq:intro-HP-IMC}
{\rm Char}_{\cO_L[[\Gamma^{\rm ac}]]}(\mathfrak{Sel}^{\bullet,\bullet}(K,\mathbf{A}^{\rm ac})^\vee_{\rm tors})\overset{?}=
{\rm Char}_{\cO_L[[\Gamma^{\rm ac}]]}\biggl(\frac{{\rm Sel}^{\bullet,\bullet}(K,\mathbf{T}^{\rm ac})}{\cO_L[[\Gamma^{\rm ac}]]\cdot\mathfrak{Z}^\bullet}\biggr)^2
\end{equation}
as ideals in $\cO_L[[\Gamma^{\rm ac}]]\otimes_{\bZ_p}\bQ_p$, 
where the subscript `tors' denotes the $\cO_L[[\Gamma^{\rm ac}]]$-torsion submodule. 
On the other hand, we prove an explicit reciprocity law for the classes $\mathfrak{Z}^\bullet$: 
\[
{\rm Log}_{\pp,{\rm ac}}^\bullet({\rm res}_\pp(\mathfrak{Z}^\bullet))=\mathscr{L}_\pp^{\tt BDP}(f/K),
\] 
relating their images under certain `signed' logarithm maps 
(whose construction is new here and might be of independent interest) 
to the $p$-adic $L$-function 
of Bertolini--Darmon--Prasanna \cite{bdp1}. 
\sk

Finally, from the explicit reciprocity laws for Beilinson--Flach classes due to Kings--Loeffler--Zerbes \cite{KLZ2}, we show that  the 
signed main conjectures in Theorem~B are closely related\footnote{The condition on the cyclotomic restriction of $\mathfrak{L}_p^{\bullet,\bullet}(f/K)$ is needed at this step of the argument.} to the Iwasawa--Greenberg main conjecture for the 
$p$-adic $L$-function $\mathscr{L}_\pp(f/K)$ constructed in \cite{wan-combined}. Since the square of $\mathscr{L}_\pp^{\tt BDP}(f/K)$ is easily seen to agree with the restriction $\mathscr{L}_\pp(f/K)$ to the `anticyclotomic line', building on our result towards  $(\ref{eq:intro-HP-IMC})$ we can thus promote the divisibility towards the Iwasawa--Greenberg main conjecture for $\mathscr{L}_\pp(f/K)$ in \cite{wan-combined} to a full equality, from which the proof of Theorem~B in the indefinite case follows easily. 
\sk

On the other hand, to prove Theorem~B in the definite case (where Heegner points are not directly available), we relate the anticyclotomic restriction of the $p$-adic $L$-functions $\mathfrak{L}_p^{\bullet,\bullet}(f/K)$ to certain signed theta elements $\Theta_\infty^\bullet\in\cO_L[[\Gamma^{\rm ac}]]$ extending a construction by Darmon--Iovita \cite{Darmon-Iovita} in the 
case $a_p=0$. Thus we are able to exploit Vatsal's results on the vanishing of anticyclotomic $\mu$-invariants \cite{vatsal-special} to deduce again from \cite{wan-combined} (under the same hypothesis on $\mathfrak{L}_p^{\bullet,\bullet}(f/K)$ as above) one of the divisibilities toward Conjecture~\ref{conj:BL-IMC}, which combined with Kato's work (as reformulated in \cite[\S{6}]{LLZ-AJM}) 
leads to the proof of Theorem~B in the definite case.

\begin{rem}
As the reader will immediately note, our proof of Theorem~A only uses Theorem~B in the definite case. On the other hand, the indefinite case of Theorem~B (and especially our progress towards the anticyclotomic main conjecture (\ref{eq:intro-HP-IMC}) on which its proof rests), 
is indispensable for the applications to the $p$-part of the Birch and Swinnerton-Dyer formula in the case of analytic rank $1$, as we shall explain in the following paragraphs. 
\end{rem}

\subsection{Applications to the $p$-part of the BSD formula}

We conclude this Introduction by explaining the implications of the preceding results of the $p$-part to the Birch and Swinnerton-Dyer formula for abelian varieties of $A/\bQ$ of ${\rm GL}_2$-type; by \cite[Cor.~10.2]{KW-SerreI} these precisely correspond to the abelian varieties $A_f/\bQ$ (up to isogeny) associated with weight $2$ eigenforms $f\in S_2(\Gamma_0(N))$ by the Eichler--Shimura construction \cite{shimura-pink}, so that
\begin{equation}\label{eq:ES}
L(A_f,s)=\prod_{\sigma}L(f^\sigma,s)\nonumber
\end{equation}
with the product running over all conjugates of $f$. In particular, letting $r_f:={\rm ord}_{s=1}L(f,s)$ we note that
\[
{\rm ord}_{s=1}L(A_f,s)=[K_f:\bQ] r_f={\rm dim}(A_f) r_f,
\]
where $K_f$ is the number field generated by the Fourier coefficients of $f$. 
\sk

The analogue of Theorem~A for primes $p\nmid N$ of ordinary reduction for $f$, i.e., Iwasawa's main conjecture for ${\rm GL}_{2/\bQ}$ for good ordinary primes, 
is one of the main results of \cite{SU}. Similarly as in [\emph{loc.cit.}, \S{3.6.1}], 
the interpolation property 
satisfied by the $p$-adic $L$-function 
$L_p^\bullet(f/\bQ)$ at the trivial character, together with a variant for the signed 
Selmer groups ${\rm Sel}^\bullet(\bQ,\mathbf{A}^{\rm cyc})$ of Mazur's control theorem \cite{greenberg-cetraro} and the relation between the specialization of ${\rm Sel}^\bullet(\bQ,\mathbf{A}^{\rm cyc})$ at the trivial character and the usual Selmer group for $A$, yields the following result towards the $p$-part of 
the Birch and Swinnerton-Dyer formula in analytic rank $0$: 

\begin{ThmC}
	Let $A/\bQ$ be a semistable abelian variety of ${\rm GL}_2$-type of conductor $N$, associated with a newform $f\in S_2(\Gamma_0(N))$, let $p\nmid N$ be a prime of good reduction for $A$, and assume that $A$ has no rational isogenies of $p$-power degree. If $L(A,1)\neq 0$, then for every prime $\mathfrak{P}$ of $K_f$ above $p$ we have
	\[
	{\rm ord}_{\mathfrak{P}}\biggl(\frac{L^{}(A,1)}{\Omega_A}\biggr)
	={\rm ord}_{\mathfrak{P}}\biggl(\#\Sha(A/\bQ)\prod_{\ell\mid N}c_\ell(A/\bQ)\biggr),
	\]
	where 
	\begin{itemize}
		\item{} $\Omega_A=\int_{A(\bR)}\vert\omega_A\vert$, for $\omega_A$ a N\'eron differential, is the real period of $A$,
		\item{} $\Sha(A/\bQ)$ is the Tate--Shafarevich group of $A$, and
		\item{} $c_\ell(A/\bQ)$ is the Tamagawa number of $A$ at the prime $\ell$.
	\end{itemize}
	In other words, the $p$-part
	of the Birch--Swinnerton-Dyer formula holds for $A$.
\end{ThmC}

On the other hand, as already indicated in the preceding paragraphs, a  key intermediate step in our proof of Theorem~B in the indefinite case is the proof of the Iwasawa--Greenberg main conjecture for the anticyclotomic $p$-adic $L$-function $\mathscr{L}_\pp^{\tt BDP}(f/K)$ (see Theorem~\ref{thm:ac-MC}). With this main conjecture in hand, an 
extension of the arguments in \cite{JSW} then yields the following result towards the $p$-part of the Birch and Swinnerton-Dyer formula in analytic rank $1$:

\begin{ThmD}\label{cor:bsd}
	Let $A/\bQ$ be a semistable abelian variety of ${\rm GL}_2$-type of conductor $N$, associated with a newform $f\in S_2(\Gamma_0(N))$. Let $p\nmid N$ be a prime of good reduction for $A$, and assume that $A$ has no rational isogenies of $p$-power degree. If ${\rm ord}_{s=1}L(f,s)=1$, 
	then for every prime $\mathfrak{P}$ of $K_f$ above $p$ we have
	\[
	{\rm ord}_{\mathfrak{P}}\biggl(\frac{L^{*}(A,1)}{{\rm Reg}(A/\bQ)\cdot\Omega_A}\biggr)
	={\rm ord}_{\mathfrak{P}}\biggl(\#\Sha(A/\bQ)\prod_{\ell\mid N}c_\ell(A/\bQ)\biggr),
	\]
	where $L^*(A,1)$ is the leading term of the Taylor series expansion of $L(A,s)$ around $s=1$, and ${\rm Reg}(A/\bQ)$ is the discriminant of the N\'eron--Tate canonical height pairing on $A(\bQ)\otimes\bR$. In other words, the $p$-part
	of the Birch--Swinnerton-Dyer formula holds for $A$.
\end{ThmD}

Finally, we note again that the signed main conjectures of \cite{BL-non-ord} are formulated for even weights $k<p$. In the \emph{definite} case, our methods should generalise, allowing one to deduce from \cite[Thm.~3.9]{wan-non-ord} a higher weight analogue of Theorem~B in this case (and yielding, in combination with \cite[Cor.~6.6]{LLZ-AJM}, a new proof of the main result of \cite{wan-non-ord}). On the other hand, a higher weight analogue of our Theorem~B in the \emph{indefinite} case would seem to require new ideas, as we would like to explore in a future work. 
\sk

\emph{Notations.} Throughout the paper, we let $f=\sum_{n=1}^\infty a_nq^n\in S_2(\Gamma_0(N))$ be a newform, $p>2$ be a good non-ordinary prime for $f$, and $K$ be an imaginary quadratic field in which $p=\pp\overline{\pp}$ splits. By choosing an algebraic isomorphism $\bC\simeq\bC_p$ once and for all, we shall assume that $f$ is defined over a finite extension $L/\bQ_p$ with ring of integers $\cO_L$. For any $p$-adic Lie group $G$, we let $\Lambda(G)$ denote the Iwasawa algebra $\cO_L[[G]]$, and set $\Lambda_L(G):=L\otimes_{\cO_L}\Lambda(G)$. Finally, for $F$ a finite extension of $\bQ$ or $\bQ_p$ we let $G_F:={\rm Gal}(\overline{F}/F)$ be the absolute Galois group. 
\sk

\emph{Acknowledgements.} We thank Kazim B{\"u}y{\"u}kboduk, Antonio Lei, David Loeffler, and Sarah Zerbes for their very helpful comments on an earlier draft of this paper.

\section{Preliminaries}\label{sec:prelim}
\subsection{Signed Coleman maps}\label{sec:Col}

In this section, we briefly recall the two-variable
signed Coleman maps introduced in \cite[\S{2.3}]{BL-non-ord}. In the notations of \emph{loc.cit.}, for our applications in this paper it will suffice to take $F=\bQ_p$, and $\vartheta$ and $\eta$ both equal to the trivial character.

Let $V_f$ be the $2$-dimensional $L$-linear $G_{\bQ}$-representation associated with $f$ by Deligne, and let $T_f\subset V_f$ be a fixed Galois stable $\cO_L$-lattice. As a representation of $G_{\bQ_p}$, $V_f$ has Hodge--Tate weights $\{0,-1\}$, with the convention that $\bQ_p(1)$ has Hodge--Tate weight $+1$. 

Let $\mathbb{A}_{\bQ_p}^+=\cO_L[[\pi]]$, where $\pi$ is a formal variable. Let $\mathbb{N}(T_f)$ be the Wach module associated to $T_f$ (see \cite[Lem.~II.1.3]{berger-limit}), and denote by $\{n_{f,1}, n_{f,2}\}$
the $\mathbb{A}_{\bQ_p}^+$-basis for $\mathbb{N}(T_f)$ constructed in
\cite[\S{3}]{BLZ}.
Since $f$ has trivial nebentypus, we have $T_f^*\simeq T_f(1)$, and therefore letting $e_1$ be a generator of $\bZ_p(1)$ and setting
\begin{equation}\label{eq:basis}
n_i:=n_{f,i}\cdot\pi^{-1}e_1,\quad i=1, 2
\end{equation}
we obtain an $\mathbb{A}_{\bQ_p}^+$-basis $\{n_1,n_2\}$
for $\mathbb{N}(T_f^*)$. By the construction in \cite{berger-limit},  
under the canonical identification
\[
\mathbb{N}(T_f^*)/\pi\mathbb{N}(T_f^*)\simeq\mathbb{D}_{\rm cris}(T_f^*), 
\]
the reduction mod $\pi$ of the basis $\{n_1,n_2\}$ yields an $\cO_L$-basis $\{v_1,v_2\}$ for $\mathbb{D}_{\rm cris}(T_f^*)$ satisfying $v_1\in{\rm Fil}^0\mathbb{D}_{\rm cris}(T_f^*)$, $v_2=\varphi(v_1)$. 
In particular, the matrix $A_\varphi$ of the Frobenius map $\varphi$ on $\mathbb{N}(T_f^*)$ with respect to $\{v_1,v_2\}$ is given by
\begin{equation}\label{eq:A}
A_\varphi=p^{-1}\left(\begin{array}{cc}0&{-1}\\p&a_p\end{array}\right).
\end{equation}

Let $\mathbb{B}_{{\rm rig},\bQ_p}^+$ be the subring of $L[[\pi]]$ of power series convergent on the open unit disc in $\bC_p$. As explained in \cite[\S{3.2}]{LLZ-AJM}, there is a change of basis matrix
$M\in M_{2\times 2}(\mathbb{B}_{{\rm rig},\bQ_p}^+)$ such that
\[
(n_1\quad n_2)=(v_1\quad v_2)\;M.
\]
Since $n_i$ reduces to $v_i$ mod $\pi$, we have 
\begin{equation}\label{eq:M-cong}
M\equiv I_2\pmod{\pi},
\end{equation}
where $I_2$ is the $2\times 2$ identity matrix. Let $\Gamma={\rm Gal}(\bQ(\mu_{p^\infty})/\bQ)$ be the Galois group of the cyclotomic  $\bZ_p^\times$-extension of $\bQ$. Like $\mathbb{A}_{\bQ_p}^+$, the ring $\mathbb{B}_{{\rm rig},\bQ_p}^+$ is equipped with $\cO_L$-linear actions of $\varphi$ and $\gamma\in\Gamma$ given by $\pi\mapsto (1+\pi)^p-1$ and $\pi\mapsto (1+\pi)^{\chi(\gamma)}-1$, where $\chi:G_{\bQ}\rightarrow\bZ_p^\times$ is the $p$-adic cyclotomic character. For any complete discretely valued extension $E$ of $\bQ_p$ let $\mathcal{H}_E(\Gamma)$ be the algebra of $E$-valued locally analytic distributions on $\Gamma$. As in \cite{LLZ-AJM}, we make the following definition.

\begin{defn}\label{def:log-matrix}
	We define the \emph{logarithmic matrix} $M_{\rm log}\in M_{2\times 2}(\mathcal{H}_L(\Gamma))$ by
	\[
	M_{\rm log}:=\mathfrak{M}^{-1}((1+\pi)A_\varphi\cdot\varphi(M)),
	\]
	where $\mathfrak{M}:\mathcal{H}_L(\Gamma)\simeq(\mathbb{B}_{{\rm rig},\bQ_p}^+)^{\psi=0}$ is the Mellin
	transform (see e.g. \cite[Cor.~B.2.8]{PR-st}), and $\psi$ is the left inverse of $\varphi$.
\end{defn}

For any $p$-adic Lie extension $E_\infty$ of $\bQ_p$, set
\[
H^1_{\rm Iw}(E_\infty,T_f^*):=\varprojlim_{F} H^1(F,T_f^*),
\]
where the limit is with respect to corestriction over the finite Galois extensions $F/\bQ_p$ contained in $E_\infty$. Let $F_\infty$ be the unramified $\bZ_p$-extension of $\bQ_p$, and set 
\[
U:={\rm Gal}(F_\infty/\bQ_p),\quad G:={\rm Gal}(F_\infty(\mu_{p^\infty})/\bQ_p)\simeq\Gamma\times U.
\] 

Let $\hat{\cO}_{F_\infty}$ be the completion of the ring of integers of $F_\infty$, and let $S_\infty\subset\Lambda_{\hat{\cO}_{F_\infty}}(U)=\hat{\cO}_{F_\infty}[[U]]$ be the Yager module introduced in \cite[\S{3.2}]{LZ2}. With a slight abuse of notation, let $1-\varphi$ be the map
\begin{equation}\label{eq:1-phi}
1-\varphi:H_{\rm Iw}^1(F_\infty(\mu_{p^\infty}),T_f^*)
\longrightarrow(\varphi^*\mathbb{N}(T_f^*))^{\psi=0}\hat{\otimes}S_\infty
\end{equation} 
constructed in [\emph{op.cit.}, Def.~4.6]. The composition of $(\ref{eq:1-phi})$ with the $\Lambda(G)$-linear embedding
\[ (\varphi^*\mathbb{N}(T_f^*))^{\psi=0}\hat{\otimes}S_\infty\hooklongrightarrow\mathcal{H}_{\hat{F}_\infty}(G)\otimes\mathbb{D}_{\rm cris}(T_f^*)
\] 
deduced from \cite[Prop.~2.11]{LLZ-ANT}, where $\hat{F}_\infty$ is the completion of $F_\infty$ and $\mathcal{H}_{\hat{F}_\infty}(G)$, yields the two-variable regulator map 
\[
\mathcal{L}_{T_f^*}:H_{\rm Iw}^1(F_\infty(\mu_{p^\infty}),T_f^*)\longrightarrow\mathcal{H}_{\hat{F}_\infty}(G)\otimes\mathbb{D}_{\rm cris}(T_f^*)
\]
of \cite{LZ2}. On the other hand, by \cite[Thm.~4.24]{LLZ-AJM} the $\mathbb{A}_{\bQ_p}^+$-basis $\{n_1,n_2\}$ of $\mathbb{N}(T_f^*)$ defined in $(\ref{eq:basis})$ is such that $\{(1+\pi)\varphi(n_1),(1+\pi)\varphi(n_2))\}$ forms
a $\Lambda(\Gamma)$-basis for $(\varphi^*\mathbb{N}(T_f^*))^{\psi=0}$, 
and so gives rise to a $\Lambda_{\hat{\cO}_{F_\infty}}(G)$-linear embedding
\[
\mathfrak{J}:(\varphi^*\mathbb{N}(T_f^*))^{\psi=0}\hat{\otimes}S_\infty\hooklongrightarrow\Lambda_{\hat{\cO}_{F_\infty}}(G)\oplus\Lambda_{\hat{\cO}_{F_\infty}}(G) 
\]
allowing us to define (following \cite[\S{2.3}]{BL-non-ord}) the two-variable signed Coleman maps
\begin{equation}\label{eq:signed-Col-Qp}
({\rm Col}^\sharp,{\rm Col}^\flat):H_{\rm Iw}^1(F_\infty(\mu_{p^\infty}),T_f^*)\longrightarrow\Lambda_{\hat{\cO}_{F_\infty}}(G)\oplus\Lambda_{\hat{\cO}_{F_\infty}}(G)
\end{equation}
as the composition $\mathfrak{J}\circ(1-\varphi)$. 

Let $\alpha$ and $\beta$ be the roots of the Hecke polynomial $X^2-a_pX+p$; since $f$ has weight $2$, we know that $\alpha\neq\beta$ by \cite{coleman-edixhoven}.  
By construction, we have the relation
\[
\mathcal{L}_{T_f^*}=(v_1\quad v_2)\cdot M_{{\rm log}}\cdot
\biggl(\begin{array}{cc}
{\rm Col}^\sharp\\
{\rm Col}^\flat
\end{array}\biggr),
\]
and letting $\mathcal{L}^\lambda$ be the projection of $\mathcal{L}_{T_f^*}$ onto the $\lambda^{-1}$-eigenspace for $\varphi$ on $\mathbb{D}_{\rm cris}(T_f^*)$ it follows that
\begin{equation}\label{eq:factor-L}
\biggl(\begin{array}{cc}
\mathcal{L}^\alpha\\
\mathcal{L}^\beta
\end{array}\biggr)
=Q_{\alpha,\beta}^{-1} M_{{\rm log}}\cdot
\biggl(\begin{array}{cc}
{\rm Col}^\sharp\\
{\rm Col}^\flat
\end{array}\biggr),
\end{equation}
where $Q_{\alpha,\beta}=\left(\begin{smallmatrix}\alpha & -\beta\\ -p&p\end{smallmatrix}\right)$ diagonalizes $A_\varphi$. 

Let $\Gamma_K:={\rm Gal}(K_\infty/K)$ be the Galois group of the  $\bZ_p^2$-extension of $K$. As in \cite[\S{2.5}]{BL-non-ord}, we shall apply the above constructions to the $G_K$-representation 
\[
\mathbf{T}:=T_f^*\hat\otimes\Lambda(\Gamma_K)^\iota,
\]
where $\Lambda(\Gamma_K)^\iota$ denotes the module $\Lambda(\Gamma_K)$ equipped with the Galois action given by the inverse of the canonical character $G_K\twoheadrightarrow\Gamma_K\hookrightarrow\Lambda(\Gamma_K)^\times$. For each $\qq\in\{\pp,\overline\pp\}$, let $D_\qq\subset\Gamma_K$ be the decomposition group of a fixed prime of $K_\infty$ above $\qq$, and let $\gamma_1,\dots,\gamma_{p^t}$ be a complete set of coset representatives for $\Gamma_K/D_\qq$, so that $\Lambda(\Gamma_K)=\sum_i\Lambda(D_\qq).\gamma_i$. Combined with Shapiro's lemma, we then have natural identifications
\[
H^1(K_\qq,\mathbf{T})\simeq\bigoplus_{i=1}^{p^t}H^1(K_\qq,T_f^*\hat{\otimes}\Lambda(D_\qq)^\iota).\gamma_i\simeq\bigoplus_{v\mid\qq}H^1_{\rm Iw}(K_{\infty,v},T_f^*),
\]
where the second sum is over the primes $v$ of $K_\infty$ above $\qq$. 

\begin{defn}\label{def:maps-q}
Let $\Delta$ be the torsion subgroup of $\Gamma$, and denote by $e_{\mathds{1}}$ is the idempotent of $\Lambda(G)$ attached to the trivial character of $\Delta$. 
\begin{enumerate}
\item{} For each $\bullet\in\{\sharp,\flat\}$ and $\qq\in\{\pp,\overline{\pp}\}$, define the \emph{signed Coleman map}
\begin{equation}\label{eq:def-Col}
{\rm Col}^\bullet_\qq:H^1(K_\qq,\mathbf{T})\longrightarrow\Lambda_{\hat{\cO}_{F_\infty}}(\Gamma_K)\nonumber
\end{equation}
by $z=\sum_iz_i.\gamma_i\mapsto\sum_ie_{\mathds{1}}{\rm Col}^\bullet(z_i).\gamma_i$, where ${\rm Col}^\bullet$ are as in (\ref{eq:signed-Col-Qp}).
\item{} For each $\bullet\in\{\sharp,\flat\}$ and $\qq\in\{\pp,\overline{\pp}\}$, define the
\emph{Perrin-Riou regulator map}
\begin{equation}\label{eq:def-L}
\mathcal{L}_\qq^\lambda:H^1(K_\qq,\mathbf{T})\longrightarrow\mathcal{H}_{\hat{F}_\infty}(\Gamma_K)\nonumber
\end{equation} 
by $z=\sum_iz_i.\gamma_i\mapsto\sum_ie_{\mathds{1}}\mathcal{L}^\lambda(z_i).\gamma_i$, where $\mathcal{L}^\lambda$ are as in $(\ref{eq:factor-L})$.
\end{enumerate}
\end{defn}

The following determination of the image of the signed Coleman maps will be important later.

\begin{prop}\label{prop:Image-Col}
The maps ${\rm Col}^\bullet_\qq$ have finite cokernel.
\end{prop}

\begin{proof}
	Since we are working in weight $2$ and projecting to the trivial isotypical component for $\Delta$, this follows from \cite[Thm.~2.12]{BL-non-ord} (see also [\emph{loc.cit.}, Rem.~2.13]). 
\end{proof}

\subsection{Signed logarithm maps}\label{sec:Log-2}
 

The two essential (pairs of) maps in this paper are the signed Coleman maps just constructed and the maps we construct in this section, the signed logarithm maps. These signed logarithm maps generalize those in the elliptic curve case (\cite{wan-combined} for $a_p=0$;  \cite{sprung-IMC} for $p>2$), but our new construction is more natural: Instead of relying on an explicit description of the kernel of the signed Coleman maps, we obtain the signed logarithm maps after relating the signed Coleman maps of $\S\ref{sec:Col}$ to a big Perrin-Riou map applied to a certain Hida family.

We maintain the notations introduced in the preceding section, and for each $\qq\in\{\pp,\overline{\pp}\}$ let $K_\infty^\qq$ be the maximal subfield of $K_\infty$ unramified at $\qq$. 
Set $\Gamma_\qq:={\rm Gal}(K_\infty^\qq/K)\simeq\bZ_p$, and let
\[
\mathbf{g}=\sum_{n=1}^\infty\mathbf{a}_nq^n\in\Lambda_{\mathbf{g}}[[q]]
\]
be the canonical Hida family of CM forms constructed in \cite[\S{5.2}]{JSW}, where $\Lambda_{\mathbf{g}}:=\Lambda(\Gamma_\pp)$. The associated Galois representation $M(\mathbf{g})^*$ satisfies
\begin{equation}\label{eq:Ind}
M(\mathbf{g})^*\vert_{G_K}\simeq\Lambda(\Gamma_{\overline\pp})^\iota\oplus\Lambda(\Gamma_\pp)^\iota,
\end{equation}
where $\Lambda(\Gamma_\qq)^\iota$ denotes the module $\Lambda(\Gamma_\qq)$ with Galois action given by the inverse of the canonical character
$G_K\twoheadrightarrow \Gamma_\qq\hookrightarrow\Lambda(\Gamma_\qq)^\times$.
Upon restriction to a decomposition group at $p$, 
$M(\mathbf{g})^*$
fits into an exact sequence of $\Lambda_{\mathbf{g}}[G_{\bQ_p}]$-modules
\begin{equation}\label{eq:fil}
0\longrightarrow\mathscr{F}^+M(\mathbf{g})^*\longrightarrow M(\mathbf{g})^*\longrightarrow\mathscr{F}^-M(\mathbf{g})^*\longrightarrow 0,
\end{equation}
with $\mathscr{F}^\pm M(\mathbf{g})^*\simeq\Lambda_{\mathbf{g}}$, and with the $G_{\bQ_p}$-action on $\mathscr{F}^-M(\mathbf{g})^*$
given by the unramified character sending a geometric Frobenius element ${\rm Fr}_p$ to $\mathbf{a}_p^{-1}$ (see 
\cite[Thm.~2.2.2]{wiles88}, or \cite[\S{7.2}]{KLZ2} and \cite{OhtaII} for the version of this result that we shall use). Letting $\mathbf{k}$ be the weight character of $\mathbf{g}$ as defined in \cite[\S{7.1}]{KLZ2}, the twist $\mathscr{F}^+M(\mathbf{g})^*(-1-\mathbf{k})$ is therefore unramified, and we set
\[
T_{f,\mathbf{g}}^+:=T_f^*\otimes_{}\mathscr{F}^+M(\mathbf{g})^*(-1-\mathbf{k}),\quad T_{f,\mathbf{g}}^-:=T_f^*\otimes_{}\mathscr{F}^-M(\mathbf{g})^*.
\]

The direct sum decomposition $(\ref{eq:Ind})$ is compatible upon restriction to a decomposition group at $p$ with the filtration $(\ref{eq:fil})$; in particular, the latter is split, so we may choose identifications\footnote{Note that the ordering of $\pp$ and $\overline{\pp}$ here is opposite
	to the one taken in \cite[\S{2.3}]{BL-non-ord}.}
\begin{equation}\label{eq:iden}
\begin{split}
\jmath_+:H^1_{\rm Iw}(\bQ_p(\mu_{p^\infty}),T_{f,\mathbf{g}}^+)&\simeq H^1(K_{\overline{\pp}},\mathbf{T}),\\
\jmath_-:H^1_{\rm Iw}(\bQ_p(\mu_{p^\infty}),T_{f,\mathbf{g}}^-)&\simeq H^1(K_{\pp},\mathbf{T}).
\end{split}
\end{equation}

For each $\lambda\in\{\alpha,\beta\}$, denote by $f^\lambda$ the $p$-stabilization of $f$ with $U_p$-eigenvalue $\lambda$. Let $\mathcal{F}^\lambda$ be the Coleman family passing through $f^\lambda$, and let $\mathcal{L}_{f^\lambda,\mathbf{g}}$ be the corresponding specialization of the composite map
\begin{align*}
\mathcal{L}_{\mathcal{F}^\lambda,\mathbf{g}}:
H^1_{\rm Iw}(\bQ(\mu_{p^\infty}),\mathscr{F}^{-+}D_{V_1\times V_2}(\mathcal{F}^\lambda\otimes\mathbf{g}))&\overset{\mathcal{L}}\longrightarrow\mathbf{D}(\mathscr{F}^{-+}M(\mathcal{F}^\lambda\otimes\mathbf{g})^*)\hat\otimes\mathcal{H}_L(\Gamma)\\
&\longrightarrow I_{\mathcal{F^\lambda}}\hat\otimes\Lambda_{\mathbf{g}}\hat\otimes\mathcal{H}_L(\Gamma),
\end{align*}
where the map $\mathcal{L}$ is the Perrin-Riou big logarithm map of \cite[Thm.~7.1.4]{LZ-Coleman}, and the second arrow is given by pairing against the tensor product $\eta_{\mathcal{F^\lambda}}\otimes\omega_{\mathbf{g}}$ of the classes constructed in [\emph{op.cit.}, \S{6.4}]. 

Let $\mathbf{D}_{\rm rig}^\dagger(V_f^*)$ and $\mathbf{D}_{\rm rig}^\dagger(M(\mathbf{g})^*)$ be the $(\varphi,\Gamma)$-module associated with $f^\lambda$ and $\mathbf{g}$, respectively, and denote by $\mathscr{F}^\pm$ the corresponding triangulations, so that $\mathscr{F}^{-+}D_{V_1\times V_2}(\mathcal{F}^\lambda\otimes\mathbf{g}))$
specializes to $\mathscr{F}^{-+}D(f^\lambda\otimes\mathbf{g}):=\mathscr{F}^{-}\mathbf{D}_{\rm rig}^\dagger(V_f^*)\otimes\mathscr{F}^+\mathbf{D}_{\rm rig}^\dagger(M(\mathbf{g})^*)$ by construction.

\begin{prop}\label{prop:comp-L}
Under the identification $\jmath^+$ in $(\ref{eq:iden})$, the $\mathcal{H}_L(\Gamma_{\overline\pp})$-linear extension of   $\mathcal{L}_\lambda^{\overline{\pp}}$ agrees with the composition of $\mathcal{L}_{f^\lambda,\mathbf{g}}$ with the projection
\begin{equation}\label{eq:fG}
\begin{split}
H^1_{\rm Iw}(\bQ_p(\mu_{p^\infty}),T_{f,\mathbf{g}}^+)\hat{\otimes}\mathcal{H}_L(\Gamma)
&\longrightarrow H^1_{\rm Iw}(\bQ_p(\mu_{p^\infty}),\mathbf{D}_{\rm rig}^\dagger(V_f^*)\otimes\mathscr{F}^+\mathbf{D}_{\rm rig}^\dagger(M(\mathbf{g})^*))\\
&\longrightarrow 
H^1_{\rm Iw}(\bQ_p(\mu_{p^\infty}),\mathscr{F}^{-+}D(f^\lambda\otimes\mathbf{g})).
\end{split}
\end{equation}
\end{prop}

\begin{proof}
By 
\cite[Cor.~4.4.11]{KPX} together with \cite[Prop.~2.4.5]{LZ-Coleman}, the first arrow in $(\ref{eq:fG})$ is an isomorphism, and the claimed agreement follows immediately by comparing the interpolation properties of the maps involved, which are given by \cite[Thm.~4.15]{LZ2} and \cite[Thm.~7.1.4]{LZ-Coleman}, respectively.
\end{proof}

Swapping the roles of $\mathcal{F}^\lambda$ and $\mathbf{g}$ in the above discussion, we will arrive at our definition of the signed logarithm maps. Indeed, denote by 
\[
\mathcal{L}_{\mathbf{g},f^\lambda}:
H^1_{\rm Iw}(\bQ_p(\mu_{p^\infty}),\mathscr{F}^{-+}D_{}(\mathbf{g}\otimes f^\lambda))\longrightarrow I_{\mathbf{g}}\hat{\otimes}\mathcal{H}_L(\Gamma)
\]
the specialization to $f^\lambda$ of the composite map
\begin{equation}\label{eq:L-F}
\begin{split}
\mathcal{L}_{\mathbf{g},\mathcal{F}^\lambda}:
H^1_{\rm Iw}(\bQ_p(\mu_{p^\infty}),\mathscr{F}^{-+}D_{V_2\times V_1}(\mathbf{g}\otimes\mathcal{F}^\lambda))&\overset{\mathcal{L}}\longrightarrow\mathbf{D}(\mathscr{F}^{-+}M(\mathbf{g}\otimes\mathcal{F}^\lambda)^*)\hat\otimes\mathcal{H}_L(\Gamma)\\
&\longrightarrow I_{\mathbf{g}}\hat\otimes\Lambda_{\mathcal{F^\lambda}}\hat\otimes\mathcal{H}_L(\Gamma),
\end{split}
\end{equation}
where $I_{\mathbf{g}}\subset{\rm Frac}(\Lambda_{\mathbf{g}})$ is the congruence ideal
for $\mathbf{g}$, and the second arrow is given by pairing against  $\eta_{\mathbf{g}}\otimes\omega_{\mathcal{F^\lambda}}$. 

\begin{defn}
	For each $\lambda\in\{\alpha,\beta\}$, let
	\[ H^1_\lambda(K_\pp,\mathbf{T}\hat{\otimes}\mathcal{H}_L(\Gamma_\pp))\subset H^1(K_\pp,\mathbf{T}\hat{\otimes}\mathcal{H}_L(\Gamma_\pp)) 
	\] 
	be the kernel of the $\mathcal{H}_L(\Gamma_\pp)$-linear extension of the map $\mathcal{L}^\lambda_\pp$, and for each $\bullet\in\{\sharp,\flat\}$, let
	\[
	H^1_{\bullet}(K_{\pp},\mathbf{T})\subset H^1(K_{\pp},\mathbf{T})
	\]
	be the kernel of the map ${\rm Col}^\bullet_\pp$.
\end{defn}

\begin{lem}\label{lem:iden-alpha}
Under the identification $\jmath^-$ in $(\ref{eq:iden})$, the image of the natural map
\[
H^1_{\rm Iw}(\bQ_p(\mu_{p^\infty}),\mathscr{F}^{-+}D_{}(\mathbf{g}\otimes f^\lambda))\hooklongrightarrow H^1_{\rm Iw}(\bQ_p(\mu_{p^\infty}),\mathscr{F}^-\mathbf{D}_{\rm rig}^\dagger(M(\mathbf{g})^*)\otimes\mathbf{D}_{\rm rig}^\dagger(V_f^*)) 
\]
is identified with $H^1_\lambda(K_\pp,\mathbf{T}\hat{\otimes}\mathcal{H}_L(\Gamma_\pp))$.
\end{lem}

\begin{proof} 
Since the regulator map $\mathcal{L}_{T_f^*}$ is injective (see \cite[Prop.~4.11]{LZ2}), a class is in the kernel of the $\mathcal{H}_L(\Gamma_\pp)$-linear extension of $\mathcal{L}_\pp^\lambda$ if and only if it corresponds under the isomorphism
\begin{align*}
H^1(K_\pp,\mathbf{T}\hat\otimes\mathcal{H}_L(\Gamma_\pp))
&\simeq 
H^1(\bQ_p(\mu_{p^\infty}),\mathscr{F}^-M(\mathbf{g})^*\otimes V_f^*\otimes\mathcal{H}_L(\Gamma))\\
&\simeq H^1_{\rm Iw}(\bQ_p(\mu_{p^\infty}),\mathscr{F}^-\mathbf{D}_{\rm rig}^\dagger(M(\mathbf{g})^*)\otimes\mathbf{D}_{\rm rig}^\dagger(V_f^*)) 
\end{align*}	
induced by $\jmath^-$ and \cite[Cor.~4.4.11]{KPX} to a class projecting trivially onto 
\[
H^1_{\rm Iw}(\bQ_p(\mu_{p^\infty}),\mathscr{F}^-\mathbf{D}_{\rm rig}^\dagger(M(\mathbf{g})^*)\otimes\mathscr{F}^-\mathbf{D}_{\rm rig}^\dagger(V_f^*)),
\]
hence the result.
\end{proof}

For every $\qq\in\{\pp,\overline{\pp}\}$, let $M_{{\rm log},\qq}\in M_{2\times 2}(\mathcal{H}_L(\Gamma_\qq))$ be the logarithmic matrix $M_{\rm log}$ of Definition~\ref{def:log-matrix} with $\Gamma_\qq$ in place of $\Gamma$.

\begin{lem}\label{eq:multi}
If $(\kappa^\sharp,\kappa^\flat)$ is any pair of classes in $H^1_\sharp(K_\pp,\mathbf{T})\oplus H^1_\flat(K_\pp,\mathbf{T})$, then the pair of classes $(\kappa^\alpha,\kappa^\beta)\in H^1(K_\pp,\mathbf{T}\hat\otimes\mathcal{H}_L(\Gamma))^{\oplus 2}$ defined by the relation
\[
\biggl(\begin{array}{cc}
\kappa^\alpha\\
\kappa^\beta
\end{array}\biggr)
=Q_{\alpha,\beta}^{-1} M_{{\rm log},\pp}\cdot
\biggl(\begin{array}{cc}
\kappa^\sharp\\
\kappa^\flat
\end{array}\biggr)
\]
lands in 
$H^1_\alpha(K_\pp,\mathbf{T}\hat\otimes\mathcal{H}_L(\Gamma_\pp))\oplus H^1_\beta(K_\pp,\mathbf{T}\hat\otimes\mathcal{H}_L(\Gamma_\pp))$.	
\end{lem}

\begin{proof}
As an immediate consequence of $(\ref{eq:factor-L})$, for each $\qq\in\{\pp,\overline{\pp}\}$ we have the factorizations
    \begin{equation}\label{eq:factor-L-q}
	\biggl(\begin{array}{cc}
	\mathcal{L}_\qq^\alpha\\
	\mathcal{L}_\qq^\beta
	\end{array}\biggr)
	=Q_{\alpha,\beta}^{-1} M_{{\rm log},\qq}\cdot
	\biggl(\begin{array}{cc}
	{\rm Col}_\qq^\sharp\\
	{\rm Col}_\qq^\flat
	\end{array}\biggr),
	\end{equation}
	which clearly implies the result.	
\end{proof}

For each $\lambda\in\{\alpha,\beta\}$ set 
\begin{equation}\label{eq:L-g}
\mathcal{L}_{\mathbf{g},\pp}^\lambda:=\mathcal{L}_{\mathbf{g},f^\lambda}\circ\jmath_-^{-1},
\end{equation}
which by Lemma~\ref{lem:iden-alpha} is naturally defined on $H^1_\lambda(K_\pp,\mathbf{T}\hat{\otimes}\mathcal{H}_L(\Gamma_\pp))$.

\begin{defn}\label{def:Log}
The two-variable \emph{signed logarithm maps}
are the maps $({\rm Log}_\pp^\sharp,{\rm Log}_\pp^\flat)$ defined by the composition
\begin{equation}\label{eq:def-Log}
\begin{split}
H^1_\sharp(K_{\pp},\mathbf{T})
\oplus H^1_\flat(K_{\pp},\mathbf{T})
&\overset{Q_{\alpha,\beta}^{-1}M_{{\rm log},\pp}}\longrightarrow
H^1_\alpha(K_\pp,\mathbf{T}\hat\otimes\mathcal{H}_{L}(\Gamma_\pp))
\oplus H^1_\beta(K_\pp,\mathbf{T}\hat\otimes\mathcal{H}_{L}(\Gamma_\pp))\\
&\overset{\mathcal{L}^\alpha_{\mathbf{g},\pp}\oplus\mathcal{L}_{\mathbf{g},\pp}^\beta}
\longrightarrow I_{\mathbf{g}}\otimes\mathcal{H}_{\hat{F}_\infty}(\Gamma_K)^{\oplus{2}}, 
\end{split}\nonumber
\end{equation}
using Lemma~\ref{lem:iden-alpha} for the first arrow.
\end{defn}

\begin{prop}\label{prop:im-Log}
The maps ${\rm Log}_\pp^\bullet$ take values in $I_{\mathbf{g}}\otimes\Lambda_{\hat{F}_\infty}(\Gamma_K)$ and have finite pseudo-null cokernel.
\end{prop}

\begin{proof}
From the construction of ${\rm Log}_\pp^\bullet$, it suffices to show that the map $\mathcal{L}_{\mathbf{g},\mathcal{F}^\lambda}$ in $(\ref{eq:L-F})$ takes values in $I_{\mathbf{g}}\hat\otimes\Lambda_{\mathcal{F}\lambda}\hat\otimes\Lambda_L(\Gamma)$ and has pseudo-null cokernel. But the first fact follows from \cite[Prop.~4.8]{LZ2} (since Frobenius acts invertibly on $\eta_{\mathbf{g}}$), and the second from [\emph{loc.cit.}, Thm.~7.1.4] (which is proved is the same manner as \cite[Thm.~8.2.3]{KLZ2}). 
\end{proof}

\subsection{Doubly-signed Selmer groups}\label{sec:Sel}

Let $\Sigma$ be a finite set of places of $K$
containing those dividing $Np\infty$, and let $\mathfrak{G}_{K,\Sigma}$ be
the Galois group of the maximal extension of $F$ unramified outside the places above $\Sigma$. Recall that module $\mathbf{T}$ introduced in $\S\ref{sec:Col}$, and set
\[
\mathbf{A}:=\mathbf{T}\otimes\Lambda(\Gamma_K)^\vee, 
\]
where $\Lambda(\Gamma_K)^\vee$ is the Pontrjagin dual of $\Lambda(\Gamma_K)$. We shall also need to consider the modules $\mathbf{T}^{\rm ac}$, $\mathbf{A}^{\rm ac}$, $\mathbf{T}^{\rm cyc}$, and $\mathbf{A}^{\rm cyc}$, obtained by replacing $\Gamma_K$ in the preceding definitions by the Galois group $\Gamma^{\rm ac}$ and $\Gamma^{\rm cyc}$ of the anticyclotomic and the cyclotomic $\bZ_p$-extension of $K$, respectively. 

In the following definitions, we let $\mathbf{M}$ denote either of the modules $\mathbf{T}$, $\mathbf{T}^{\rm ac}$, or $\mathbf{T}^{\rm cyc}$. 

\begin{defn}\label{def:Sel}
	The $p$-relaxed Selmer group of $\mathbf{M}$ is
	\begin{equation}
	\mathfrak{Sel}^{\{p\}}(K,\mathbf{M}):={\rm ker}\Biggl\{H^1(\mathfrak{G}_{K,\Sigma},\mathbf{M})
	\longrightarrow\bigoplus_{v\in\Sigma\smallsetminus\{p\}}\frac{H^1(K_v,\mathbf{M})}{H^1_{\rm ur}(K_v,\mathbf{M})}\Biggr\},\nonumber
	\end{equation}
	where 
	\[
	H^1_{\rm ur}(K_v,\mathbf{M}):=
	{\rm ker}\{H^1(K_v,\mathbf{M})\longrightarrow H^1(K_v^{\rm ur},\mathbf{M})\}
	\]
	is the unramified local condition. On the other hand, the $p$-related \emph{strict} Selmer group of $\mathbf{M}$ is 
	\begin{equation}
	{\rm Sel}^{\{p\}}(K,\mathbf{M}):={\rm ker}\Biggl\{H^1(\mathfrak{G}_{K,\Sigma},\mathbf{M})
	\longrightarrow\bigoplus_{v\in\Sigma\smallsetminus\{p\}}H^1(K_v,\mathbf{M})\Biggr\}.\nonumber
	\end{equation}
\end{defn}

Our Selmer groups of interest in this paper are obtained from cutting the $p$-relaxed ones by various local conditions at the primes above $p$.

\begin{defn}
	For $\qq\in\{\pp,\overline{\pp}\}$ and $\mathscr{L}_\qq\in\{{\rm rel},{\rm str},\sharp,\flat\}$, set
	\[
	H^1_{\mathscr{L}_\qq}(K_\qq,\mathbf{M})=
	\left\{
	\begin{array}{ll}
	H^1(K_\qq,\mathbf{M})&\textrm{if $\mathscr{L}_\qq={\rm rel}$,}\\
	\{0\}&\textrm{if $\mathscr{L}_\qq={\rm str}$,}\\
	{\rm ker}({\rm Col}^\bullet_\qq)&\textrm{if $\mathscr{L}_\qq=\bullet\in\{\sharp,\flat\}$,}\\
	\end{array}
	\right.
	\]
	where ${\rm Col}^\bullet_\qq$ is the signed Coleman map of (\ref{eq:def-Col}) (or its anticyclotomic or cyclotomic projection),
	and for $\mathscr{L}=\{\mathscr{L}_\pp,\mathscr{L}_{\overline{\pp}}\}$, define
	\begin{equation}
	\mathfrak{Sel}^{\mathscr{L}}(K,\mathbf{M}):={\rm ker}\Biggr\{{\rm Sel}^{\{p\}}(K,\mathbf{M})
	\longrightarrow\bigoplus_{\qq\in\{\pp,\overline{\pp}\}}\frac{H^1(K_\qq,\mathbf{M})}{H_{\mathscr{L}_{\qq}}^1(K_\qq,\mathbf{M})}\Biggr\},\nonumber
	\end{equation}
	and similarly for ${\rm Sel}^{\mathscr{L}}(K,\mathbf{M})$
\end{defn}

Thus, for example, ${\rm Sel}^{{\rm rel},{\rm str}}(K,\mathbf{M})$ is the submodule of ${\rm Sel}^{\{p\}}(K,\mathbf{M})$ consisting of classes which are trivial at $\overline{\pp}$ (with no condition at $\pp$). 

On the other hand, letting $\mathbf{W}$ denote either of the modules $\mathbf{A}$, $\mathbf{A}^{\rm ac}$, or $\mathbf{A}^{\rm cyc}$, we define $\mathfrak{Sel}^{\{p\}}(K,\mathbf{W})$ and ${\rm Sel}^{\{p\}}(K,\mathbf{W})$ just as in Definition~\ref{def:Sel}, and set
\begin{equation}
\mathfrak{Sel}^{\mathscr{L}}(K,\mathbf{W}):={\rm ker}\Biggr\{\mathfrak{Sel}^{\{p\}}(K,\mathbf{W})
\longrightarrow\bigoplus_{\qq\in\{\pp,\overline{\pp}\}}\frac{H^1(K_\qq,\mathbf{W})}{H_{\mathscr{L}_{\qq}}^1(K_\qq,\mathbf{M})^\perp}\Biggr\},\nonumber
\end{equation}
where $H_{\mathscr{L}_{\qq}}^1(K_\qq,\mathbf{M})^\perp$ is the orthogonal complement of $H_{\mathscr{L}_{\qq}}^1(K_\qq,\mathbf{M})$ under local Tate duality, and define ${\rm Sel}^{\mathscr{L}}(K,\mathbf{W})\subset {\rm Sel}^{\{p\}}(K,\mathbf{W})$ in the same manner. Finally, we let 
\[
\mathfrak{X}^{\mathscr{L}}_{K_\infty}(f):={\rm Hom}_{\bZ_p}(\mathfrak{Sel}^{\mathscr{L}}(K,\mathbf{A}),\bQ_p/\bZ_p)
\]
be the Pontrjagin dual of $\mathfrak{Sel}^{\mathscr{L}}(K,\mathbf{A})$, and similarly define $\mathfrak{X}^{\mathscr{L}}_{K^{\rm ac}_\infty}(f)$, $\mathfrak{X}^{\mathscr{L}}_{K^{\rm cyc}_\infty}(f)$; and their strict analogues $X^{\mathscr{L}}_{K_\infty}(f)$, $X^{\mathscr{L}}_{K^{\rm ac}_\infty}(f)$, and $X^{\mathscr{L}}_{K^{\rm cyc}_\infty}(f)$.

\begin{rem} \label{rem:Sel-Q}
Later in the paper (see esp. $\S\S{5.2}$-3) we shall also need to consider signed Selmer groups ${\rm Sel}^\bullet(\bQ,\mathbf{T}^{\rm cyc})$ and $X_{\bQ_\infty}^\bullet(f)$ for the cyclotomic $\bZ_p$-extension of $\bQ$, whose definition (with the strict condition at the places outside $p$) is the obvious one from the above.
\end{rem}

\section{Two-variable Iwasawa theory}
\subsection{Signed Beilinson--Flach classes}\label{sec:BF}

We keep the notations introduced in the $\S\ref{sec:prelim}$. As in \cite[\S{2.2}]{LZ-Coleman}, for each $h\in\mathbf{R}_{\geqslant 0}$ we let $D_h(\Gamma,T_f^*)$ denote the space of $T_f^*$-valued distributions on $\Gamma$ of order $h$, so that in particular we have
\[
H^1(\bQ,D_0(\Gamma,T_f^*)\hat\otimes M(\mathbf{g})^*)\simeq 
H^1_{\rm Iw}(\bQ(\mu_{p^\infty}),T_f^*\hat\otimes M(\mathbf{g})^*)[1/p]
\]
by \cite[Prop.~2.4.2]{LZ-Coleman}. For each $\lambda\in\{\alpha,\beta\}$, set $v_\lambda:={\rm ord}_p(\lambda)$ and denote by
\begin{equation}\label{eq:BF-col}
\mathcal{BF}^{\lambda,\mathbf{g}}\in H^1(\bQ,D_{v_\lambda}(\Gamma,T_f^*)\hat\otimes M(\mathbf{g})^*)
\end{equation}
the class ${\rm BF}^{\lambda,\mathbf{g}}_m$ 
of  \cite[Thm.~3.2]{BL-non-ord} for $m=1$; up to a $p$-adic multiplier, this corresponds to the image of the Beilinson--Flach class $_{c}\mathcal{BF}^{\mathcal{F}^\lambda,\mathbf{g}}_{1,0}$ 
of \cite[\S{5.4}]{LZ-Coleman} under the map on cohomology induced by the specialization $\mathcal{F}^\lambda\to f^\lambda$,  with the auxiliary factor $c$ disposed of. 

By Shapiro's lemma and \cite[Prop.~2.4.2]{LZ-Coleman}, the classes $(\ref{eq:BF-col})$ may be identified with classes
\[
\mathcal{BF}^{\lambda}\in H^1(K,\mathbf{T})
\hat\otimes\mathcal{H}_{L,v_\lambda}(\Gamma_{\pp}),
\]
where we recall that $\mathbf{T}:=T_f^*\hat{\otimes}\Lambda(\Gamma_K)^\iota$.

\begin{thm}[B\"uy\"ukboduk--Lei]\label{thm:BF-factor}
There exist classes  
\[
\mathcal{BF}^{\sharp},\;\mathcal{BF}^{\flat}
\in 
H^1(K,\mathbf{T})[1/p]
\] 
such that
\begin{equation}
\left(\begin{array}{cc}
\mathcal{BF}^{\alpha}\\
\mathcal{BF}^{\beta}
\end{array}\right)
=Q_{\alpha,\beta}^{-1}M_{{\rm log},\pp}\cdot
\left(\begin{array}{cc}
\mathcal{BF}^{\sharp}\\
\mathcal{BF}^{\flat}
\end{array}\right),\nonumber
\end{equation}
where $Q_{\alpha,\beta}=\left(\begin{smallmatrix}\alpha & -\beta\\ -p&p\end{smallmatrix}\right)$ and $M_{{\rm log},\pp}\in M_{2\times 2}(\mathcal{H}_L(\Gamma_\pp))$ is the logarithmic matrix of Definition~\ref{def:log-matrix} with $\Gamma_\pp$ in place of $\Gamma$.
\end{thm}

\begin{proof}
This follows from \cite[Thm.~3.7]{BL-non-ord} and the discussion in [\emph{loc.cit.}, $\S{3.3}$].
\end{proof}

In the following, we fix one for all a nonzero element $c\in\cO_L$ such that the classes
\[
\mathcal{BF}_c^\bullet:=c\cdot\mathcal{BF}^\bullet
\]
land in $H^1(K,\mathbf{T})$ for both $\bullet\in\{\sharp,\flat\}$. (The particular choice of $c$ is irrelevant, since it will eventually cancel out with its contribution elsewhere.)



\subsection{Explicit reciprocity laws, I}\label{sec:Log} 

In this section, we consider the local images at $\overline{\pp}$ of the Beilinson--Flach classes. 

By \cite[Thm.~7.1.2]{LZ-Coleman}, for each $\lambda\in\{\alpha,\beta\}$ the image of ${\rm res}_p(\mathcal{BF}^{\lambda,\mathbf{g}})$ under the composite map
\begin{align*}
H^1(\bQ_p,D_{v_\lambda}(\Gamma,T_f^*)\hat{\otimes}M(\mathbf{g})^*)&\longrightarrow H^1_{\rm Iw}(\bQ_p(\mu_{p^\infty}),\mathbf{D}_{\rm rig}^\dagger(V_f^*)\hat{\otimes}\mathbf{D}_{\rm rig}^\dagger(M(\mathbf{g})^*))\\
&\longrightarrow H^1_{\rm Iw}(\bQ_p(\mu_{p^\infty}),\mathscr{F}^-\mathbf{D}_{\rm rig}^\dagger(V_f^*)\hat{\otimes}\mathbf{D}_{\rm rig}^\dagger(M(\mathbf{g})^*)),
\end{align*}
where the first arrow is given by \cite[Cor.~6.1.4]{LZ-Coleman} and the second one in the natural projection, lies in the image of the natural map
\[
H^1_{\rm Iw}(\bQ_p(\mu_{p^\infty}),\mathscr{F}^{-+}D(f\otimes\mathbf{g}))\longrightarrow H^1_{\rm Iw}(\bQ_p(\mu_{p^\infty}),\mathscr{F}^-\mathbf{D}_{\rm rig}^\dagger(V_f^*)\hat{\otimes}\mathbf{D}_{\rm rig}^\dagger(M(\mathbf{g})^*)).
\]
Thus letting ${\rm res}_{\overline\pp}(\mathcal{BF}^\lambda)\in H^1(K_{\overline{\pp}},\mathbf{T})\hat{\otimes}\mathcal{H}_{L,v_\lambda}(\Gamma_{\pp})$ be the image of ${\rm res}_p(\mathcal{BF}^{\lambda,\mathbf{g}})$ under the identification $\jmath^+$ of $(\ref{eq:iden})$, by Proposition~\ref{prop:comp-L} we may consider the element
\begin{equation}\label{eq:BF-L}
\mathfrak{L}_{p}^{\lambda,\mu}(f/K):=\mathcal{L}^\mu_{\overline{\pp}}({\rm res}_{\overline{\pp}}(\mathcal{BF}^\lambda))
\end{equation}
in $\mathcal{H}_{L,v_\lambda}(\Gamma_{\pp})\hat{\otimes}\mathcal{H}_{\hat{F}_\infty,v_\mu}(\Gamma_{\overline\pp})\subseteq\mathcal{H}_{\hat{F}_\infty}(\Gamma_K)$ for each $\mu\in\{\alpha,\beta\}$. 

\begin{rem}
For $\bullet,\circ\in\{\sharp,\flat\}$, setting 
\begin{equation}\label{eq:BF-L-signed}
\mathfrak{L}_p^{\bullet,\circ}(f/K):={\rm Col}^\circ_{\overline\pp}({\rm res}_{\overline{\pp}}(\mathcal{BF}_c^\bullet)),
\end{equation}
the factorizations in $(\ref{eq:factor-L-q})$ and Theorem~\ref{thm:BF-factor} imply that
\begin{equation}\label{eq:factor-2-Lp}
\biggl(\begin{array}{cc}c\mathfrak{L}_{p}^{\alpha,\alpha}(f/K)&c\mathfrak{L}_{p}^{\beta,\alpha}(f/K)\\
c\mathfrak{L}_{p}^{\alpha,\beta}(f/K)&c\mathfrak{L}_{p}^{\beta,\beta}(f/K)\end{array}\biggr)
=Q_{\alpha,\beta}^{-1}M_{{\rm log},\overline\pp}
\cdot\biggl(\begin{array}{cc}\mathfrak{L}_p^{\sharp,\sharp}(f/K)&\mathfrak{L}_p^{\flat,\sharp}(f/K)\\
\mathfrak{L}_p^{\sharp,\flat}(f/K)&\mathfrak{L}_p^{\flat,\flat}(f/K)\end{array}\biggr)
\cdot(Q_{\alpha,\beta}^{-1}M_{{\rm log},\pp})^\intercal,
\end{equation}
where $A^\intercal$ denotes the transpose of a matrix $A$.
\end{rem}

By the following explicit reciprocity law from the work of Loeffler--Zerbes \cite{LZ-Coleman}, the $p$-adic $L$-functions $(\ref{eq:BF-L})$, at least when $\lambda=\mu$, interpolate critical values for the Rankin--Selberg 
$L$-function of $f$ twisted by finite order Hecke characters of $K$. 

\begin{thm}[Loeffler--Zerbes]\label{thm:BF-ERL}
For every $\lambda\in\{\alpha,\beta\}$, if $\chi:\Gamma_K\rightarrow\mu_{p^\infty}$ is any
finite order character of conductor $\mathfrak{c}_\chi$, then
\begin{align*}
\chi(\mathfrak{L}_{p}^{\lambda,\lambda}(f/K))&=
\Biggl(\prod_{\qq\mid p}\lambda^{-v_\qq(\mathfrak{c}_\chi)}\Biggr)
\cdot\frac{\mathcal{E}(f,\chi)}{\mathfrak{g}(\chi)\cdot \vert\mathfrak{c}_\chi\vert^{1/2}}
\cdot\frac{L_{}(f/K,\chi,1)}{(4\pi)^2\cdot \langle f,f\rangle_N},
\end{align*}
where
\[
\mathcal{E}(f,\chi)=\prod_{\qq\mid p,\;\qq\nmid\mathfrak{c}_\chi}(1-\lambda^{-1}\chi(\qq))
(1-\lambda^{-1}\chi^{-1}(\qq)).
\]
\end{thm}

\begin{proof}
This follows from \cite[Thm.~7.1.5]{LZ-Coleman}, as complemented by \cite{loeffler-note}.
\end{proof}

\begin{cor}\label{cor:Rohrlich-K}
	For every $\lambda\in\{\alpha,\beta\}$, the $p$-adic $L$-function $\mathfrak{L}_p^{\lambda,\lambda}(f/K)$ is nonzero.
\end{cor}

\begin{proof}
Immediate from Theorem~\ref{thm:BF-ERL} and Rohrlich's nonvanishing results (see \cite{rohrlich-K}).
\end{proof}


Let $\Gamma^{\rm cyc}$ be the Galois group of the cyclotomic $\bZ_p$-extension of $K$, which we shall also see as the Galois group of the cyclotomic $\bZ_p$-extension of $\bQ$, and 
denote by $\mathfrak{L}_p^{\bullet,\circ}(f/K)_{\rm cyc}$ the image of $\mathfrak{L}_p^{\bullet,\circ}(f/K)$
under the natural projection $\Lambda(\Gamma_K)\rightarrow\Lambda(\Gamma^{\rm cyc})$. By \cite[Thm.~3.25]{LLZ-AJM}, there exist elements $L_p^\sharp(f), L_p^\flat(f)\in\Lambda_L(\Gamma^{\rm cyc})$ such that
\begin{equation}\label{eq:LLZ-Q}
\left(\begin{array}{cc}
L_{p}^\alpha(f)\\
L_{p}^\beta(f)
\end{array}\right)
=Q_{\alpha,\beta}^{-1} M_{\rm log}\cdot
\left(\begin{array}{cc}
L_{p}^\sharp(f)\\
L_p^\flat(f)
\end{array}\right),
\end{equation}
where $L_{p}^\alpha(f), L_p^\beta(f)\in\mathcal{H}_{L}(\Gamma^{\rm cyc})$
are the (unbounded) $p$-adic $L$-functions of Amice--V\'elu and Vishik (see \cite{mtt} or \cite[\S{2}]{pollack}).

\begin{prop}\label{prop:cyc-comp}
For each pair $\bullet,\circ\in\{\sharp,\flat\}$ we have
\[
\mathfrak{L}_{p}^{\bullet,\circ}(f/K)_{\rm cyc}=\frac{c}{2}(L_p^\bullet(f)\cdot L_p^\circ(f_K)+L_p^\circ(f)\cdot L_p^\bullet(f_K)),
\] 
where $f_K:=f\otimes\epsilon_K$ is the twist of $f$ by the quadratic character associated to $K$.
\end{prop}


\begin{proof}[Sketch of proof]
Denote by $\mathfrak{L}_{p}^{\lambda,\mu}(f/K)_{\rm cyc}$  the image of the element $\mathfrak{L}_p^{\lambda,\mu}(f/K)$ of $(\ref{eq:BF-L})$ under the natural map induced by the projection $\Gamma_K\twoheadrightarrow\Gamma^{\rm cyc}$. From their respective interpolation properties, we see that
\begin{equation}\label{eq:factor-unb}
\mathfrak{L}_p^{\lambda,\mu}(f/K)_{\rm cyc}=\frac{1}{2}(L_{p}^\lambda(f)\cdot L_{p}^\mu(f_K)+L_{p}^\mu(f)\cdot L_{p}^\lambda(f_K)).
\end{equation}

Indeed, for $\lambda=\mu$ this is shown by proving the analogous factorization for the two-variable $p$-adic $L$-functions attached to the Coleman family passing through the $p$-stabilization of $f$ of $U_p$-eigenvalue $\lambda$, using that for higher weights we have more critical values to compare the two sides, and specializing back to $f$. (See \cite[Lem.~4.22]{wan-non-ord}, noting that the constant in \emph{loc.cit.} is given by the ratio $\Omega_f^{\rm can}/\Omega_f^+\Omega_f^-$, which in our weight $2$ case is a $p$-adic unit by \cite[Lem.~9.5]{skinner-zhang}.) 
	
For $\lambda\neq\mu$, one uses the existence of an element (constructed in ongoing work of B{\"u}y{\"u}kboduk--Lei--Loeffler--Venkat, specializing \cite[Thm.~3.7.1]{BLLV}) in $\bigwedge^2 H^1_{\rm Iw}(K_\infty,V_f^*)$ whose restriction to 
\[
\bigwedge^2 H^1_{\rm Iw}(K^{\rm cyc}_\infty,V_f^*)\simeq H^1_{\rm Iw}(\bQ^{\rm cyc},V_f^*)\otimes H^1_{\rm Iw}(\bQ^{\rm cyc},V_{f_K}^*)
\]
	agrees (by $(\ref{eq:factor-unb})$ in the cases $\lambda=\mu$) with the tensor product $\mathbf{z}_f\otimes\mathbf{z}_{f_K}$ of Kato's zeta elements. This agreement, together with Kato's explicit reciprocity law yields the asymmetric cases of $(\ref{eq:factor-unb})$. 
	
	Letting
	\[
	\mathcal{L}:=\biggl(\begin{array}{cc}
	L_{p}^\alpha(f) L_{p}^\alpha(f_K)&L_{p}^\alpha(f)L_{p}^\beta(f_K)\\
	L_p^\beta(f) L_{p}^\alpha(f_K)&L_{p}^\beta(f) L_{p}^\beta(f_K)\end{array}\biggr),
	\]
	by the decomposition $(\ref{eq:factor-2-Lp})$, we thus have
	\begin{equation}\label{eq:factor-Lp}
	Q_{\alpha,\beta}^{-1}M_{{\rm log}}
	\cdot\biggl(\begin{array}{cc}\mathfrak{L}_p^{\sharp,\sharp}(f/K)_{\rm cyc}&\mathfrak{L}_p^{\flat,\sharp}(f/K)_{\rm cyc}\\
	\mathfrak{L}_p^{\sharp,\flat}(f/K)_{\rm cyc}&\mathfrak{L}_p^{\flat,\flat}(f/K)_{\rm cyc}\end{array}\biggr)
	\cdot(Q_{\alpha,\beta}^{-1}M_{{\rm log}})^\intercal=\frac{c}{2}(\mathcal{L}+\mathcal{L}^\intercal).
	\end{equation}
	
	On the other hand, since $p$ splits in $K$, we have $\epsilon_K(p)=1$,
	so the roots of the $p$-th Hecke of polynomial of $f_K$ are the same as that of $f$.
	Thus taking the product of the matrices in the factorization $(\ref{eq:LLZ-Q})$ and
	the analogous factorization for $f_K$:
	\begin{equation}\label{eq:LLZ-Q-K}
	\left(\begin{array}{c}L_p^\alpha(f_K)\\L_p^\beta(f_K)\end{array}\right)^\intercal=
	\left(\begin{array}{c}L_p^\sharp(f_K)\\L_p^\flat(f_K)\end{array}\right)^\intercal\cdot
	(Q_{\alpha,\beta}^{-1}M_{\rm log})^\intercal,
	\end{equation}
	we see that the above matrix $\mathcal{L}$ is also given by 
	\begin{equation}
	\begin{split}\label{eq:tensor}
	\mathcal{L}&=Q_{\alpha,\beta}^{-1}M_{{\rm log}}
	\cdot\biggl(\begin{array}{cc}
	L_p^{\sharp}(f) L_p^{\sharp}(f_K)&L_p^{\sharp}(f) L_p^{\flat}(f_K)\\
	L_p^{\flat}(f) L_p^{\sharp}(f_K)&L_p^{\flat}(f) L_p^{\flat}(f_K)\end{array}\biggr)
	\cdot(Q_{\alpha,\beta}^{-1}M_{{\rm log}})^\intercal,
	\end{split}
	\end{equation}
	from where the result follows from substituting the expression for $\mathcal{L}$ in $(\ref{eq:tensor})$ into $(\ref{eq:factor-Lp})$.	
\end{proof}

\subsection{Explicit reciprocity laws, II} 

In this section, we consider the images of the Beilinson--Flach classes under the restriction map at $\pp$.  

As shown in \cite[Cor.~3.15]{BL-non-ord}, for each $\bullet\in\{\sharp,\flat\}$
we have the inclusion
\begin{equation}\label{eq:BF-respp}
{\rm res}_\pp(\mathcal{BF}_c^\bullet)\in {\rm ker}({\rm Col}^\bullet_\pp);
\end{equation}
in particular, the signed logarithm map ${\rm Log}^{\bullet}_\pp$
introduced in $\S\ref{sec:Log-2}$ may be applied to this class.

Following the terminology introduced in \cite[Prop.~8.3]{wan-combined}, we say that a height one prime $P\subset\Lambda(\Gamma_K)$ is ``exceptional'', if $P$ is the pullback of the augmentation ideal $(\gamma_{\ac}-1)\Lambda(\Gamma^{\rm ac})$ of $\Lambda(\Gamma^{\rm ac})$. 
From work of Hida--Tilouine \cite{HT-117} and Rubin \cite{rubin-IMC} (see e.g. \cite[Thm.~2.7]{cas-BF}) one can show that the congruence ideal $I_{\mathbf{g}}$
appearing in $(\ref{eq:L-F})$ is generated by $\frac{h_K}{w_K}\cdot L_\pp^{\rm ac}(K)$ up to exceptional primes, where $h_K$ is the class number of $K$, $w_K$ the number of roots of unity of $K$, and
$L_\pp^{\rm ac}(K)$ an anticyclotomic Katz $p$-adic $L$-function. Thus 
for each $\bullet\in\{\sharp,\flat\}$ the map
\[
\widetilde{\rm Log}^{\bullet}_\pp:=\frac{h_K}{w_K}\cdot L_\pp^{\rm ac}(K)\times{\rm Log}^{\bullet}_\pp
\]
is integral up to exceptional primes.

For the next result, let $\bar{\rho}_f:G_\bQ\rightarrow{\rm GL}_2(\kappa_L)$ be the residual representation associated with $f$, where $\kappa_L$ is the residue field of $L$. 

\begin{thm}\label{thm:ERL-pp}
Assume that $\bar{\rho}_f$ is ramified at every prime $\ell\mid N$ which is nonsplit in $K$. Then for each $\bullet\in\{\sharp,\flat\}$ we have
\[
\widetilde{\rm Log}_{\pp}^\bullet({\rm res}_\pp(\mathcal{BF}_c^\bullet))=c\mathscr{L}_\pp(f/K)
\]
up to exceptional primes, where $\mathscr{L}_\pp(f/K)$ is the two-variable Rankin--Selberg $p$-adic $L$-function
constructed in \cite[\S{4.6}]{wan-combined}.
\end{thm}

\begin{proof}
For each $\lambda\in\{\alpha,\beta\}$, let $L_{p,\lambda}(f/K)\in{\rm Frac}(\mathcal{H}_L(\Gamma_K))$ be the image of Urban's three-variable $p$-adic $L$-function
$L_p(\mathbf{g},\mathcal{F}^\lambda,1+\mathbf{j})$ under the specialization map $\mathcal{F}^\lambda\rightarrow f^\lambda$, where as always  $\mathcal{F}^\lambda$ denotes the Coleman family passing through to $f^\lambda$. By the explicit reciprocity law of \cite[Thm.~7.1.5]{LZ-Coleman} and \cite{loeffler-note} we have the relation
\begin{equation}\label{eq:ERL}
\mathcal{L}_{\mathbf{g},\pp}^\lambda({\rm res}_\pp(\mathcal{BF}_c^\lambda))=cL_{p,\lambda}(f/K),
\end{equation}
where $\mathcal{L}_{\mathbf{g},\pp}^\lambda$ is as in $(\ref{eq:L-g})$. By Definition~\ref{def:Log}
and the factorization of Theorem~\ref{thm:BF-factor}, the pair of equalities $(\ref{eq:ERL})$ for $\lambda=\alpha$ and $\beta$ amounts to the equality
\begin{equation}\label{eq:vector}
\biggl(\begin{array}{c}
{\rm Log}^{\sharp}_\pp({\rm res}_\pp(\mathcal{BF}_c^\sharp))\\
{\rm Log}^{\flat}_\pp({\rm res}_\pp(\mathcal{BF}_c^\flat))
\end{array}\biggr)
=\left(\begin{array}{c}
cL_{p,\alpha}(f/K)\\
cL_{p,\beta}(f/K)
\end{array}
\right).
\end{equation}

On the other hand, by the calculations in \cite[\S{5.3}]{JSW} we have the relation
\begin{equation}\label{eq:factor-w}
\mathscr{L}_\pp(f/K)=\frac{h_K}{w_K}\cdot L_\pp^{\rm ac}(K)\cdot L_{p,\lambda}(f/K),
\end{equation}
possibly up to powers of $p$ coming the constant $\alpha(f,f_B)$ appearing in [\emph{loc.cit.}, $\S5.1$]. Since by the discussion in \cite[p.~912]{prasanna} our ramification hypothesis on $\bar{\rho}_f$ forces $\alpha(f,f_B)$ to be a $p$-adic unit, the result follows from $(\ref{eq:vector})$ and $(\ref{eq:factor-w})$.
\end{proof}

\subsection{Main conjectures}\label{sec:equiv}

In this section, building on the explicit reciprocity laws of the preceding sections, we relate different variants of the two-variable Iwasawa main conjectures for modular forms at non-ordinary primes.

\begin{thm}\label{thm:equiv}
The following three statements are equivalent, where the equalities of characteristic ideals are up to exceptional primes.
\begin{enumerate}
\item{}
Both ${\rm Sel}^{{\rm str},{\rm rel}}(K,\mathbf{T})$ and
$X^{{\rm rel},{\rm str}}_{K_\infty}(f)$ are $\Lambda(\Gamma_K)$-torsion, and
\[
{\rm Char}_{\Lambda(\Gamma_K)}(X^{{\rm rel},{\rm str}}_{K_\infty}(f))=(\mathscr{L}_\pp(f/K))
\]
as ideals in $\Lambda_{R_0}(\Gamma_K)$.
\item{}
For all $\bullet\in\{\sharp,\flat\}$,
$X_{K_\infty}^{\bullet,{\rm str}}(f)$ is $\Lambda(\Gamma_K)$-torsion,
${\rm Sel}^{\bullet,{\rm rel}}(K,\mathbf{T})$ has $\Lambda(\Gamma_K)$-rank $1$, and
\[
c\cdot{\rm Char}_{\Lambda(\Gamma_K)}(X^{\bullet,{\rm str}}_{K_\infty}(f))
={\rm Char}_{\Lambda(\Gamma_K)}\biggl(\frac{{\rm Sel}^{\bullet,{\rm rel}}(K,\mathbf{T})}{\Lambda(\Gamma_K)\cdot\mathcal{BF}_c^\bullet}\biggr)
\]
as ideals in $\Lambda(\Gamma_K)$.
\item{} 
If $\bullet,\circ\in\{\sharp,\flat\}$ are such that $\mathfrak{L}_p^{\bullet,\circ}(f/K)$ is nonzero, then $X^{\bullet,\circ}_{K_\infty}(f)$ is $\Lambda(\Gamma_K)$-torsion, and
\[
c\cdot{\rm Char}_{\Lambda(\Gamma_K)}(X^{\bullet,\circ}_{K_\infty}(f))=(\mathfrak{L}_p^{\bullet,\circ}(f/K))
\]
as ideals in $\Lambda(\Gamma_K)$.
\end{enumerate}
\end{thm}

\begin{proof}
We shall just prove the implications $(1)\Rightarrow(2)\Rightarrow(3)$,
since the converse implications are shown in the same way. Poitou--Tate duality gives rise to the exact sequence
\begin{equation}\label{eq:PT1}
0\longrightarrow{\rm Sel}^{{\rm str},{\rm rel}}(K,\mathbf{T})\longrightarrow{\rm Sel}^{\bullet,{\rm rel}}(K,\mathbf{T})
\overset{{\rm res}_\pp}\longrightarrow H^1_\bullet(K_\pp,\mathbf{T})\longrightarrow
X^{{\rm rel},{\rm str}}_{K_\infty}(f)\longrightarrow X^{\bullet,{\rm str}}_{K_\infty}(f)\longrightarrow 0.
\end{equation}

Assume that $(1)$ holds. Since $\mathscr{L}_\pp(f/K)$ is nonzero (see \cite[Rem.~1.3]{cas-BF}), by
Theorem~\ref{thm:ERL-pp} the image of ${\rm res}_\pp$ is not $\Lambda(\Gamma_K)$-torsion,
and since $H^1_\bullet(K_\pp,\mathbf{T})$ has $\Lambda(\Gamma_K)$-rank $1$, it follows from $(\ref{eq:PT1})$
that ${\rm Sel}^{\bullet,{\rm rel}}(K,\mathbf{T})$ has also   $\Lambda(\Gamma_K)$-rank $1$. Since $\bar{\rho}_f$ is irreducible by \cite{Edi}, the $\Lambda(\Gamma_K)$-torsion in $H^1(K,\mathbf{T})$ is zero (see e.g. \cite[Lem.~2.6]{cas-BF}), and so
our assumption implies that 
${\rm Sel}^{{\rm str},{\rm rel}}(K,\mathbf{T})=\{0\}$. From $(\ref{eq:PT1})$ we thus arrive at the exact sequence
\begin{equation}\label{eq:PT1-div}
0\longrightarrow\frac{{\rm Sel}^{\bullet,{\rm rel}}(K,\mathbf{T})}{\Lambda(\Gamma_K)\cdot\mathcal{BF}_c^\bullet}
\overset{}\longrightarrow\frac{{\rm Im}(\widetilde{\rm Log}_\pp^\bullet)}{(c\cdot\mathscr{L}_\pp(f/K))}
\longrightarrow X^{{\rm rel},{\rm str}}_{K_\infty}(f)\longrightarrow X^{\bullet,{\rm str}}_{K_\infty}(f)\longrightarrow 0,
\end{equation}
using Theorem~\ref{thm:ERL-pp} for the second term; in particular, it follows that $X^{\bullet,{\rm str}}_{K_\infty}(f)$ is $\Lambda(\Gamma_K)$-torsion. Since the map $\widetilde{\rm Log}_\pp^\bullet$ has
pseudo-null cokernel by Proposition~\ref{prop:im-Log}, the conclusion in part (2)
follows after taking characteristic ideals in $(\ref{eq:PT1-div})$.

Assume now that the hypotheses in $(2)$ hold, let $\bullet, \circ\in\{\sharp,\flat\}$ be such that $\mathfrak{L}_p^{\bullet,\circ}(f/K)$ is nonzero, and consider the exact sequence
\begin{equation}\label{eq:PT2}
0\longrightarrow{\rm Sel}^{\bullet,\circ}(K,\mathbf{T})\longrightarrow{\rm Sel}^{\bullet,{\rm rel}}(K,\mathbf{T})
\overset{{\rm res}_{\overline\pp}}\longrightarrow\frac{H^1(K_{\overline\pp},\mathbf{T})}{H^1_\circ(K_{\overline\pp},\mathbf{T})}
\longrightarrow X^{\bullet,\circ}_{K_\infty}(f)\longrightarrow X^{\bullet,{\rm str}}_{K_\infty}(f)\longrightarrow 0.
\end{equation}
Since $\mathcal{BF}^\bullet$ lands in ${\rm Sel}^{\bullet,{\rm rel}}(K,\mathbf{T})$ by \cite[Cor.~3.12]{BL-non-ord} 
and $\mathfrak{L}_p^{\bullet,\circ}(f/K)$ is nonzero by hypothesis, the image of ${\rm res}_{\overline{\pp}}$ is not $\Lambda(\Gamma_K)$-torsion (see (\ref{eq:BF-L-signed})), 
and since $H^1(K_{\overline\pp},\mathbf{T})/H^1_\circ(K_{\overline\pp},\mathbf{T})$ has $\Lambda(\Gamma_K)$-rank $1$, we deduce
from $(\ref{eq:PT2})$ that both ${\rm Sel}^{\bullet,\circ}(K,\mathbf{T})$ and $X^{\bullet,\circ}_{K_\infty}(f)$
are $\Lambda(\Gamma_K)$-torsion. Since ${\rm Sel}^{\bullet,\circ}(K,\mathbf{T})$ is then forced to vanish
similarly as before, we arrive at the exact sequence
\begin{equation}\label{eq:PT2-div}
0\longrightarrow\frac{{\rm Sel}^{\bullet,{\rm rel}}(K,\mathbf{T})}{\Lambda(\Gamma_K)\cdot\mathcal{BF}_c^\bullet}
\overset{}\longrightarrow\frac{{\rm Im}({\rm Col}_{\overline\pp}^\circ)}{(\mathfrak{L}_p^{\bullet,\circ}(f/K))}
\longrightarrow X^{\bullet,\circ}_{K_\infty}(f)\longrightarrow X^{\bullet,{\rm str}}_{K_\infty}(f)\longrightarrow 0,
\end{equation}
using $(\ref{eq:BF-L-signed})$ for the second term. Since ${\rm Col}_{\overline\pp}^\circ$ has
pseudo-null cokernel by Proposition~\ref{prop:Image-Col}, the conclusion in part (3) now follows after taking characteristic ideals in (\ref{eq:PT2-div}). 
\end{proof}

\section{Anticyclotomic Iwasawa theory}\label{sec:ac-Iw}

Throughout this section, we let $f=\sum_{n=1}^\infty a_nq^n\in S_2^{\rm new}(\Gamma_0(N))$ be a newform, and $p>2$ be a prime of good non-ordinary reduction for $f$ as before. The imaginary quadratic field $K$ (in which we continue to assume that $p=\pp\overline{\pp}$ split) determines a factorization:
\begin{equation}\label{eq:gen-Heeg}
\begin{split}
\bullet\; & \textrm{$N=N^+N^-$ with $(N^+,N^-)=1$},\\
\bullet\; & \textrm{$\ell\mid N^+$ if and only if $\ell$ is split or ramified in $K$},\\
\bullet\; & \textrm{$\ell\mid N^-$ if and only if $\ell$ is inert in $K$}.\\
\end{split}\nonumber 
\end{equation}

We shall assume that   
\begin{equation}
\textrm{$N^-$ is square-free,} \nonumber
\end{equation}
and say that the pair $(f,K)$ is \emph{indefinite} (resp. \emph{definite}) if $N^-$
is the product of an even (resp. odd) number of primes;
if $\pi$ is the cuspidal automorphic representation of ${\rm GL}_2(\mathbb{A})$ associated with $f$ and $\epsilon(\pi,K,s)$ is the epsilon-factor of the base change of $\pi$ to ${\rm GL}_2(\mathbb{A}_K)$,
we then have $\epsilon(\pi,K,\frac{1}{2})=-1$ (resp. $\epsilon(\pi,K,\frac{1}{2})=+1$).

\subsection{Signed Heegner points}\label{sec:HP}

Assume that $(f,K)$ is an indefinite pair in the above sense, so that $N^-$ is the square-free product of an \emph{even} number of primes. Let $X=X_{N^+,N^-}$ be the Shimura curve (with the cusps added if $N^-=1$)
over $\bQ$ attached to the quaternion algebra $B/\bQ$ of discriminant $N^-$ and an Eichler order $R\subset\cO_B$ of level $N^+$.
We embed $X$ into its Jacobian $J(X)$ by choosing
an auxiliary prime $\ell_0\nmid Np$ and defining
\[
\iota_{\ell_0}:X\longrightarrow J(X)
\]
by $x\mapsto(T_{\ell_0}-\ell_0-1)[x]$, where $T_{\ell_0}$ is the usual Hecke correspondence on $X$,
and $[x]\in{\rm Div}(X)$ is the divisor class of $x\in X$. Attached to $f$ is an isogeny class of abelian variety quotients: 
\begin{equation}\label{eq:param}
\pi_f:J(X)\longrightarrow A_f.
\end{equation}

Let $\kappa_L$ be the residue field of $L$, and denote by $\bar{\rho}_f:G_\bQ\rightarrow{\rm GL}_2(\kappa_L)$ the reduction of $\rho_f$ modulo the maximal ideal of $\cO_L$. Letting $K_f\subseteq\bC$ be the Hecke field of $f$, and $\cO_f\subseteq K_f$ be its ring of integers, we may assume $L=K_{f,\mathfrak{P}}$, where $\mathfrak{P}$ is the prime of $K_f$ above $p$ induced by our fixed isomorphism $\bC\simeq\bC_p$. Then $T_f^*$ corresponds to the $\mathfrak{P}$-adic Tate module of $A_f$, and up to changing $A_f$ within its isogeny class, we may and will assume that $\cO_f\hookrightarrow{\rm End}_\bQ(A_f)$. 

For each integer $m>0$, let $\cO_m=\bZ+m\cO_K$ be the order of $K$ of conductor
$m$, and denote by $K[m]$ the ring class field of $K$ of that conductor, so
that ${\rm Gal}(K[m]/K)\simeq{\rm Pic}(\cO_m)$ under the Artin reciprocity map.

\begin{prop}\label{prop:Heeg}
There is a collection of Heegner points $h[m]\in A_f(K[m])$, indexed by the
positive integers $m$ prime to $ND_K$, such that
\[
[\cO_{m}^\times:\cO_{m\ell}^\times]\cdot{\rm Norm}^{K[m\ell]}_{K[m]}(h[m\ell])=
\left\{
\begin{array}{ll}
a_\ell\cdot h[m] & \textrm{if $\ell\nmid m$ is inert in $K$,}\\
(a_\ell-\sigma_{\mathfrak{l}}-\sigma_{\overline{\mathfrak{l}}})\cdot h[m]
&\textrm{if $\ell=\mathfrak{l}\overline{\mathfrak{l}}\nmid m$ splits in $K$},\\
a_\ell\cdot h[m]-h[m/\ell]&\textrm{if $\ell\mid m$,}
\end{array}
\right.
\]
where $\sigma_{\mathfrak{l}},\sigma_{\overline{\mathfrak{l}}}\in{\rm Gal}(K[m]/K)$
are the Frobenius elements at $\mathfrak{l}, \overline{\mathfrak{l}}$, respectively.
\end{prop}

\begin{proof}
This is standard, letting $h([m]))$ be the
image under the composition
\[
X\overset{\iota_{\ell_0}}\longrightarrow J(X)\overset{\pi_f}\longrightarrow A_f
\]
of certain canonical
CM points $\tilde{h}([m])\in X(K[m])$, see e.g. \cite[Prop.~1.3.2]{howard-PhD-II}.
\end{proof}

From now on, we choose the prime $\ell_0$ above so that $a_{\ell_0}-\ell_0-1$ is a unit in $\cO_L^\times$,
and let
\[
y[m]\in H^1(K[m],T_f^*)
\]
be the image of $h[m]\otimes(a_{\ell_0}(f)-\ell_0-1)^{-1}\in A_f(K[m])\otimes\cO_L$
under the Kummer map. (Note that $y[m]$ is independent of the choice of $\ell_0$ by construction.)

For every integer $m>0$ prime to $NDp$, let $K_n^{\rm ac}[m]$ denote the compositum
$K_n^{\rm ac}K[m]$, and define
\begin{equation}\label{eq:Heeg-not-stab}
\mathfrak{Y}_n[m]:={\rm cor}^{K[mp^{k(n)}]}_{K_n^{\rm ac}[m]}(y[mp^{k(n)}]),
\end{equation}
where $k(n):={\rm min}\{k\;\vert\;K_n^{\rm ac}\subset K[p^k]\}$.

\begin{defn}\label{def:Heeg-stab}
Let $\lambda\in\{\alpha,\beta\}$ be a root of 
$X^2-a_pX+p$, and let $m>0$ be an integer prime to $NDp$.
The \emph{$\lambda$-stabilized Heegner class}
\[
\mathfrak{Z}_n[m]^\lambda\in H^1(K_n^{\rm ac}[m],T_f^*)
\] 
is defined by
\[
\mathfrak{Z}_{n}[m]^\lambda:=
\left\{
\begin{array}{ll}
\mathfrak{Y}_n[m]-\frac{1}{\lambda}\cdot\mathfrak{Y}_{n-1}[m]&\textrm{if $n>0$,}\\
\frac{1}{u_K}\left(1-\frac{1}{\lambda}\sigma_\pp\right)\left(1-\frac{1}{\lambda}\sigma_{\overline\pp}\right)
\cdot\mathfrak{Y}_0[m]&\textrm{if $n=0$,}
\end{array}
\right.
\]
where $u_K:=\vert\cO_K^\times\vert$/2.
\end{defn}

Letting ${\rm cor}^{n}_{n-1}$ denote the corestriction map for the extension $K_{n}^{\rm ac}[m]/K_{n-1}^{\rm ac}[m]$,
an immediate calculation using Proposition~\ref{prop:Heeg} reveals that
\begin{equation}\label{eq:cor-Heeg}
{\rm cor}^{n}_{n-1}(\mathfrak{Z}_{n}[m]^\lambda)=\lambda\cdot\mathfrak{Z}_{n-1}[m]^\lambda
\end{equation}
for all $n>0$. Let $\Gamma^{\rm ac}={\rm Gal}(K_\infty^{\rm ac}/K)$ be the Galois group of the anticyclotomic $\bZ_p$-extension of $K$, and set
\[
\mathbf{T}^{\rm ac}:=T_f^*\hat\otimes\Lambda(\Gamma^{\rm ac})^\iota,
\]
where $\Lambda(\Gamma^{\rm ac})^\iota$ is the module $\Lambda(\Gamma^{\rm ac})$ equipped with the  $G_K$-action given by the inverse of the tautological character $G_K\twoheadrightarrow\Gamma^{\rm ac}\hookrightarrow\Lambda(\Gamma^{\rm ac})^\times$. 

\begin{prop}\label{prop:unb-Heeg}
For each $\lambda\in\{\alpha,\beta\}$ there exists a unique element
\[
\mathfrak{Z}[m]^\lambda\in H^1_{}(K[m],\mathbf{T}^{\rm ac})
\hat\otimes_{}\mathcal{H}_{L,v_\lambda}(\Gamma^{\rm ac}),
\]
where $v_\lambda:={\rm ord}_p(\lambda)$, whose image in $H^1(K_n^{\rm ac}[m],V_f^*)$ is $\lambda^{-n}\cdot\mathfrak{Z}_{n}[m]^\lambda$ for all $n\geqslant 0$.
\end{prop}

\begin{proof}
Since $f$ has weight $2$ and $p$ is non-ordinary for $f$, we have $v_\lambda<1$, and the result follows from $(\ref{eq:cor-Heeg})$
and \cite[Prop.~A.2.10]{LLZ}.
\end{proof}

In order to construct from the pair of unbounded Heegner classes of Proposition~\ref{prop:unb-Heeg} a
pair of ``signed'' Heegner classes with bounded growth over the anticyclotomic tower, we will need
the following variant of a useful lemma from \cite{BL-non-ord}. Let $\gamma_{\rm ac}\in\Gamma^{\rm ac}$ be a topological generator, and denote by $M_{{\rm log},{\rm ac}}\in M_{2\times 2}(\mathcal{H}_L(\Gamma^{\rm ac}))$ the logarithmic matrix $M_{\rm log}$ of Definition~\ref{def:log-matrix} with $\Gamma^{\rm ac}$ in place of $\Gamma$.

\begin{lem}\label{lem:coeffs}
Suppose that for every $\lambda\in\{\alpha,\beta\}$ there exists
$F^\lambda\in\mathcal{H}_{L,v_\lambda}(\Gamma^{\rm ac})$
such that for all finite order character $\phi$ of $\Gamma^{\rm ac}$ of conductor $p^n>1$, we have
\[
\phi(F^\lambda)=\lambda^{-n}\cdot c_{\phi}
\]
for some $c_\phi\in\overline{\bQ}_p$ independent of $\lambda$. Then there exist elements
$F^\sharp, F^\flat\in\Lambda_L(\Gamma^{\rm ac})$ such that
\[
\left(\begin{array}{cc}
F^{\alpha}\\
F^{\beta}
\end{array}\right)
=Q_{\alpha,\beta}^{-1}M_{{\rm log},{\rm ac}}\cdot
\left(\begin{array}{cc}
F^{\sharp}\\
F^{\flat}
\end{array}\right).\nonumber
\]
Moreover, if there exists a sequence of polynomials $P_{n,\lambda}\in\cO_L[(\gamma_{\rm ac}-1)]$ such that
\[
F_\lambda\equiv\lambda^{-n}\cdot P_{\lambda,n}\pmod{\gamma_{\rm ac}^{p^n}-1},
\]
then there exists a nonzero $C\in\cO_L$ depending only on $\lambda$ such that $F^\sharp, F^\flat\in C^{-1}\Lambda(\Gamma^{\rm ac})$.
\end{lem}

\begin{rem}
	For the sake of clarity, let us remark that the last statement in Lemma~\ref{lem:coeffs} means that if $G^\alpha$ and $G^\beta$ are \emph{any} other pair of elements as in the statement, the corresponding $G^\sharp, G^\bullet$ produced by the lemma will be in $C^{-1}\Lambda(\Gamma^{\rm ac})$ for \emph{the same} nonzero $C\in\cO_L$.
\end{rem}

\begin{proof}
See \cite[Prop.~2.10]{BL-non-ord}, whose proof with $\Gamma^{\rm ac}$ in place of $\Gamma$ is the same.
\end{proof}


\begin{thm}\label{thm:Heeg}
Assume that 
$\bar{\rho}_f\vert_{G_K}$ is irreducible.
Then for each positive integer $m$ prime to $NDp$ there exist bounded classes 
\[
\mathfrak{Z}[m]^{\sharp},\;\mathfrak{Z}[m]^{\flat}
\in H^1(K[m],\mathbf{T}^{\rm ac})[1/p]
\] 
such that
\begin{equation}\label{eq:BF}
\left(\begin{array}{cc}
\mathfrak{Z}[m]^{\alpha}\\
\mathfrak{Z}[m]^{\beta}
\end{array}\right)
=Q_{\alpha,\beta}^{-1}M_{{\rm log},{\rm ac}}\cdot
\left(\begin{array}{cc}
\mathfrak{Z}[m]^{\sharp}\\
\mathfrak{Z}[m]^{\flat}
\end{array}\right).\nonumber
\end{equation}
Moreover, there is a nonzero $C\in\cO_L$ such that $\mathfrak{Z}[m]^\sharp, \mathfrak{Z}[m]^\flat\in C^{-1}H^1(K[m],\mathbf{T}^{\rm ac})$ for all $m$.
\end{thm}

\begin{proof}
This follows from a straightforward adjustment of the argument in \cite[Thm.~3.7]{BL-non-ord}.
Indeed, by \cite[Lem.~4.3]{cas-wan-ss} the assumption on $\bar{\rho}_f$ implies that
$H^1(K[m],\mathbf{T}^{\rm ac})$ is free over $\Lambda(\Gamma^{\rm ac})$,
and so $H^1_{}(K[m],\mathbf{T}^{\rm ac})\hat\otimes_{}\mathcal{H}_{L,v_\lambda}(\Gamma^{\rm ac})$ 
is free over $\mathcal{H}_{L,v_\lambda}(\Gamma^{\rm ac})$.
Writing
\begin{equation}\label{eq:expand}
\mathfrak{Z}[m]^\lambda=\sum_iF_{i,m}^\lambda\cdot\mathfrak{Z}[m]_i\nonumber
\end{equation}
with $F_{i,m}^\lambda\in\mathcal{H}_{L,v_\lambda}(\Gamma^{\rm ac})$ and
$\mathfrak{Z}[m]_i\in H^1_{}(K,\mathbf{T}^{\rm ac})$, we see from
Proposition~\ref{prop:unb-Heeg} and the definition of $\mathfrak{Z}_n[m]^\lambda$
that if $\phi:\Gamma^{\rm ac}\rightarrow\mu_{p^\infty}$ is any finite order character of
$\Gamma^{\rm ac}$ of conductor $p^n>1$, then
\[
\phi(F_{i,m}^\lambda)=\lambda^{-n}\cdot c_\phi^\lambda
\]
for some $c_\phi^\lambda\in\overline{\bQ}_p$ with $c_\phi^\alpha=c_\phi^\beta$.
By Lemma~\ref{lem:coeffs} applied to the coefficients $F_{i,m}^\lambda$, the existence of classes $\mathfrak{Z}[m]^\bullet $ decomposing the unbounded $\mathfrak{Z}[m]^\lambda$ as indicated follows. Moreover, since  $\mathfrak{Z}_n[m]^\lambda$ clearly lands in $\lambda^{-1}H^1(K_n^{\rm ac},T_f^*)$, the existence of  polynomials $P_{n,\lambda}$ as in Lemma~\ref{lem:coeffs} for each $F_{i,m}^\lambda$ follows again from Proposition~\ref{prop:unb-Heeg}, and therefore the last property of the classes $\mathfrak{Z}[m]^\bullet$ follows from the last claim in Lemma~\ref{lem:coeffs}.
\end{proof}

From now on, we fix a nonzero $c\in\cO_L$ as in Theorem~\ref{thm:Heeg}, and set
\[
\mathfrak{Z}_c[m]^\bullet:=c\cdot\mathfrak{Z}[m]^\bullet\in H^1(K[m],\Tc)
\]
for each $m$ prime to $NDp$ and $\bullet\in\{\sharp,\flat\}$.

\subsection{Explicit reciprocity law}\label{sec:HP-ERL}

We keep the hypotheses from the preceding section, let $\unr$ denote the completion of the ring of integers of
the maximal unramified extension of $\bQ_p$.

\begin{thm}\label{thm:bdp}
There exists an element $\mathscr{L}^{\tt BDP}_{\pp}(f/K)\in\Lambda_{\unr}(\Gamma^{\ac})$ such that if
$\psi:\Gamma^\ac\rightarrow\bC_p^\times$ has trivial conductor and
infinity type $(-m,m)$ with $m>0$, then
\begin{align*}
\biggl(\frac{\psi(\mathscr{L}^{\tt BDP}_{\pp}(f/K))}{\Omega_p^{2m}}\biggr)^2&=
\Gamma(m)\Gamma(m+1)
\cdot(1-p^{-1}\psi(\pp)\alpha)^2(1-p^{-1}\psi(\pp)\beta)^2
\cdot\frac{L(f/K,\psi,1)}{\pi^{2m+1}\cdot\Omega_K^{4m}},
\end{align*}
where 
$(\Omega_p,\Omega_K)\in\unr^\times\times\bC^\times$ are CM periods attached to $K$.
Moreover, $\mathscr{L}_{\pp}^{\tt BDP}(f/K)$ is nonzero, and if
$\bar{\rho}_f\vert_{G_K}$ is absolutely irreducible, the $\mu$-invariant of $\mathscr{L}_{\pp}^{\tt BDP}(f/K)$ vanishes.
\end{thm}

\begin{proof}
The construction of $\mathscr{L}_{\pp}^{\tt BDP}(f/K)$ 
follows from the results in \cite[\S{3.3}]{cas-hsieh1} (see the proof of Theorem~\ref{thm:ERL-bdp} below
for the precise relation between $\mathscr{L}_{\pp}^{\tt BDP}(f/K)$ and the $p$-adic $L$-function constructed in \emph{loc.cit.}).
The nontriviality of $\mathscr{L}_{\pp}^{\tt BDP}(f/K)$ is deduced in \cite[Thm.~3.7]{cas-hsieh1}
as a consequence of \cite[Thm.~C]{hsieh}, and the vanishing of its $\mu$-invariant
similarly follows from \cite[Thm.~B]{hsieh} (and from \cite[Thm.~B]{burungale-II}, in the cases
where $N^-\neq 1$).
\end{proof}

As in the proof of \cite[Thm.~5.1]{cas-hsieh1}, one may deduce from
the two-variable regulator maps $\mathcal{L}^\lambda_\pp$ in Definition~\ref{def:maps-q} the construction of linear maps
\begin{equation}\label{eq:PR-K-ac}
\mathcal{L}_{\pp,{\rm ac}}^\lambda:H^1(K_\pp,\mathbf{T}^{\rm ac})
\longrightarrow\mathcal{H}_{\hat{F}_\infty}(\Gamma^{\rm ac})
\end{equation}
such that for any class $\mathbf{z}\in H^1(K_\pp,\mathbf{T}^{\rm ac})$
and $\chi$ a finite order character of $\Gamma^{\rm ac}$, we have
\begin{equation}\label{eq:interp-exp}
\chi(\mathcal{L}_{\pp,{\rm ac}}^\lambda(\mathbf{z}))\doteq
\langle{\rm exp}^*(\mathbf{z}^{\chi^{-1}}),\eta_{f^\lambda}\rangle,
\end{equation}
where ${\rm exp}^*$ is the Bloch--Kato dual exponential map for the twisted representation
$V_f^*(\chi^{-1})$, $\mathbf{z}^{\chi^{-1}}$ is the natural image
of $\mathbf{z}$ in $H^1(K_\pp,V_f^*(\chi^{-1}))$, and
$\eta_{f^\lambda}$ is the specialization of the class $\eta_{\mathcal{F}^\lambda}$ appeared right before Proposition~\ref{prop:comp-L}. 
The same construction applied to the two-variable Coleman maps ${\rm Col}^\bullet_\pp$ in Definition~\ref{def:maps-q} yields $\Lambda(\Gamma^{\rm ac})$-linear maps 
\[
{\rm Col}_{\pp,{\rm ac}}^\bullet:H^1(K_\pp,\mathbf{T}^{\rm ac})
\rightarrow\Lambda_{\hat{\cO}_{F_\infty}}(\Gamma^{\rm ac})
\] 
satisfying
\begin{equation}\label{eq:Col-ac}
\left(\begin{array}{cc}
\mathcal{L}^\alpha_{\pp,{\rm ac}}\\
\mathcal{L}^\beta_{\pp,{\rm ac}}
\end{array}\right)
=Q_{\alpha,\beta}^{-1}M_{{\rm log},{\rm ac}}\cdot
\biggl(\begin{array}{cc}
{\rm Col}^\sharp_{\pp,{\rm ac}}\\
{\rm Col}^\flat_{\pp,{\rm ac}}
\end{array}\biggr)
\end{equation}
as an immediate consequence of $(\ref{eq:factor-L-q})$.

\begin{lem}\label{lem:Heeg-Col}
For each $\bullet\in\{\sharp,\flat\}$ and $m$ a positive integer prime to $NDp$, we have
\[
{\rm Col}^\bullet_{\pp,{\rm ac}}({\rm res}_\pp(\mathfrak{Z}_c[m]^\bullet))=0.
\]
\end{lem}

\begin{proof}
The following argument is inspired by the proof of \cite[Cor.~3.15]{BL-non-ord}. 
We only prove the case $\bullet=\sharp$, since the proof for the other case is essentially the same.
By their construction as Kummer images of Heegner points, 
if $\mathfrak{P}$ is any prime of $K_n^{\rm ac}[m]$ above $\pp$,
the $\lambda$-stabilized Heegner classes $\mathfrak{Z}_n[m]^\lambda$ of Definition~\ref{def:Heeg-stab} have ${\rm res}_{\mathfrak{P}}(\mathfrak{Z}_n[m]^\lambda)$ in the kernel of the Bloch--Kato
dual exponential map. By $(\ref{eq:interp-exp})$ and Proposition~\ref{prop:unb-Heeg},
it follows that
\begin{equation}\label{eq:L-0}
\mathcal{L}_{\pp,{\rm ac}}^\mu({\rm res}_\pp(\mathfrak{Z}[m]^\lambda))=0
\end{equation}
for all $\lambda,\mu\in\{\alpha,\beta\}$. Write $Q_{\alpha,\beta}^{-1}M_{{\rm log},{\rm ac}}
=\left(\begin{smallmatrix}a&b\\c&d\end{smallmatrix}\right)$, and set
\begin{equation}
\begin{split}\label{eq:Z-hat}
\hat{\mathfrak{Z}}[m]^\sharp:=&\;d\cdot\mathfrak{Z}[m]^\alpha-b\cdot\mathfrak{Z}[m]^\beta\\
=&\;{\rm det}(Q_{\alpha,\beta}^{-1}M_{{\rm log},{\rm ac}})\cdot\mathfrak{Z}[m]^\sharp
\in H^1(K[m],\mathbf{T}^{\rm ac})\hat\otimes\mathcal{H}_{L}(\Gamma^{\rm ac}),
\end{split}
\end{equation}
where the second equality follows from the decomposition in Theorem~\ref{thm:Heeg}.
Extending $(\ref{eq:PR-K-ac})$ by $\mathcal{H}_{L}(\Gamma^{\rm ac})$-linearity, we deduce
from $(\ref{eq:L-0})$ that
\[
\mathcal{L}_{\pp,{\rm ac}}^{\mu}({\rm res}_\pp(\hat{\mathfrak{Z}}[m]^\sharp))=0
\]
for all $\mu\in\{\alpha,\beta\}$, and similarly extending the maps ${\rm Col}_{\pp,{\rm ac}}^\bullet$
we conclude that
\begin{equation}\label{eq:Col-hat}
{\rm Col}_{\pp,{\rm ac}}^\sharp({\rm res}_\pp(\hat{\mathfrak{Z}}[m]^\sharp))=0
\end{equation}
by the decomposition $(\ref{eq:Col-ac})$. Combining $(\ref{eq:Z-hat})$ and $(\ref{eq:Col-hat})$,
the result follows.
\end{proof}

\begin{prop}\label{prop:L-ac}
For each $\lambda\in\{\alpha,\beta\}$, there exists an injective linear map
\[
\mathscr{L}_{\pp,\ac}^\lambda:H^1_{}(K_{\pp},\mathbf{T}^{\rm ac}\hat{\otimes}\mathcal{H}_{L}(\Gamma^{\rm ac}))
\longrightarrow\mathcal{H}_{\hat{F}_\infty}(\Gamma^\ac)
\]
with finite cokernel such that for every $\mathbf{z}\in H^1_{}(K_{\pp},\mathbf{T}^{\rm ac}\hat{\otimes}\mathcal{H}_{L}(\Gamma^{\ac}))$
and $\phi:\Gamma^\ac\rightarrow\overline{\bQ}_p^\times$ of infinity type $(-\ell,\ell)$
with $\ell>0$, we have
\[
\phi(\mathscr{L}_{\pp,\ac}^\lambda(\mathbf{z}))=\frac{\mathfrak{g}(\phi^{-1})\phi^{-1}(p^n)}{\lambda^np^n}
\cdot\frac{(-1)^{\ell-1}}{(\ell-1)!}\cdot\langle{\rm log}_{V_f}(\mathbf{z}^{\phi^{-1}}),\omega_{f^\lambda}\otimes t^\ell\rangle,
\]
where $\mathfrak{g}(\phi^{-1})$ is a Gauss sum, ${\rm log}_{V_f}$
is the Bloch--Kato logarithm, and $t\in\mathbf{B}_{\rm dR}$ is Fontaine's
$2\pi i$.
\end{prop}

\begin{proof}
A straighforward modification of the proof of \cite[Thm.~5.1]{cas-hsieh1}, replacing the appeal to \cite[Thm.~4.7]{LZ2}
by the analogous construction based on \cite[(6.2.1)]{LZ-Coleman}.
\end{proof}

Let $H_\bullet^1(K_\pp,\mathbf{T}^{\rm ac})\subset H^1(K_\pp,\mathbf{T}^{\rm ac})$
be the kernel of the map ${\rm Col}_{\pp,{\rm ac}}^\bullet$, so that in particular by Lemma~\ref{lem:Heeg-Col} we have
\[
{\rm res}_\pp(\mathfrak{Z}_c[1]^\bullet)\in H^1_\bullet(K_\pp,\mathbf{T}^{\rm ac})
\]
for every $\bullet\in\{\sharp,\flat\}$. 
Letting 
$H_\lambda^1(K_\pp,\mathbf{T}^{\rm ac}\hat\otimes\mathcal{H}_L(\Gamma^{\rm ac}))$ be the kernel of the $\mathcal{H}_L(\Gamma^{\rm ac})$-linear extension of the map $\mathcal{L}^\lambda_{\pp,{\rm ac}}$, we clearly have an anticyclotomic analogue of Lemma~\ref{lem:iden-alpha}, and following Definition~\ref{def:Log}, we define the anticyclotomic signed logarithm  maps $({\rm Log}^\sharp_{\pp,\ac},{\rm Log}^\flat_{\pp,\ac})$ by the composition
\begin{align*}
H^1_\sharp(K_{\pp},\mathbf{T}^\ac)
\oplus H^1_\flat(K_{\pp},\mathbf{T}^\ac)
&\overset{Q_{\lambda,\mu}^{-1}M_{{\rm log},\ac}}\longrightarrow
H^1_{\alpha}(K_\pp,\mathbf{T}^\ac\hat\otimes\mathcal{H}_{L}(\Gamma^{\ac}))\oplus H^1_{\beta}(K_\pp,\mathbf{T}^\ac\hat\otimes\mathcal{H}_{L}(\Gamma^{\ac}))\\
&\overset{\mathscr{L}^\lambda_{\pp,\ac}\oplus\mathscr{L}_{\pp,\ac}^\mu}\longrightarrow
\mathcal{H}_{L}(\Gamma^\ac)^{\oplus{2}},
\end{align*}
where $\mathscr{L}_{\pp,{\rm ac}}^\lambda$ and $\mathscr{L}_{\pp,{\rm ac}}^\mu$ are as in
Proposition~\ref{prop:L-ac}.

\begin{thm}\label{thm:ERL-bdp}
For each $\bullet\in\{\sharp,\flat\}$, the class $\mathfrak{Z}_c^\bullet:=\mathfrak{Z}_c[1]^\bullet$ is such that
\[
{\rm Log}_{\pp,{\rm ac}}^\bullet({\rm res}_\pp(\mathfrak{Z}_c^\bullet))=c\cdot\mathscr{L}_{\pp}^{\tt BDP}(f/K)
\]
up to a unit, where $\mathscr{L}_{\pp}^{\tt BDP}(f/K)$ is as in Theorem~\ref{thm:bdp}.
In particular, the class $\mathfrak{Z}_c^\bullet$ is nontorsion over $\Lambda(\Gamma^{\rm ac})$.
\end{thm}

\begin{proof}
Let $\psi$ be an anticyclotomic Hecke character of $K$ 
of infinity type $(1,-1)$
and conductor prime to $p$, and let $\mathscr{L}_{\pp,\psi}(f)\in\Lambda_{\unr}(\Gamma^{\rm ac})$ be as
in \cite[Def.~3.7]{cas-hsieh1}. The $p$-adic $L$-function $\mathscr{L}_{\pp}^{\tt BDP}(f/K)$
of Theorem~\ref{thm:bdp} is then given by
\[
\mathscr{L}_{\pp}^{\tt BDP}(f/K)={\rm Tw}_{\psi^{-1}}(\mathscr{L}_{\pp,\psi}(f)),
\]
where ${\rm Tw}_{\psi^{-1}}:\Lambda_{\unr}(\Gamma^{\rm ac})\rightarrow\Lambda_{\unr}(\Gamma^{\rm ac})$
is the $\unr$-linear isomorphism given by $\gamma\mapsto\psi^{-1}(\gamma)\gamma$
for $\gamma\in\Gamma^{\ac}$. Let $\phi:\Gamma^{\ac}\rightarrow\mu_{p^\infty}$
be a nontrivial character of conductor $p^n$. Following the calculations in \cite[Thm.~4.9]{cas-hsieh1},
we see that
\begin{equation}\label{eq:calc}
\phi(\mathscr{L}_\pp^{\tt BDP}(f/K))=\mathfrak{g}(\phi)\phi^{-1}(p^n)p^{-n}
\sum_{\sigma\in\Gamma^{\ac}/{p^n}}\phi^{-1}(\sigma)\cdot\langle\log_{V_f}({\rm res}_\pp(\mathfrak{Y}_n[1]^\sigma)),\omega_f\otimes t\rangle,
\end{equation}
where $\mathfrak{Y}_n[1]$ is as in $(\ref{eq:Heeg-not-stab})$. (Note that the calculation in \cite[Thm.~4.9]{cas-hsieh1} is done under the hypothesis ({\rm Heeg}') in \emph{loc.cit.}, 
but its extension to the case considered here is a straightforward modification of Brook's extension \cite[Thm.~8.11]{brooks} of \cite[Thm.~5.13]{bdp1}; see also \cite[Thm.~6.1]{burungale-II}.)
Since $n>0$, we may replace $\mathfrak{Y}_n[1]$ by any of its `$p$-stabilizations' $\mathfrak{Z}_n[1]^\lambda$; 
combined with Proposition~\ref{prop:unb-Heeg} and Proposition~\ref{prop:L-ac},
we thus conclude from $(\ref{eq:calc})$ that for each $\lambda\in\{\alpha,\beta\}$ we have
\begin{equation}\label{eq:unb-ERL}
\mathscr{L}_{\pp,{\rm ac}}^\lambda({\rm res}_\pp(\mathfrak{Z}^\lambda))=\mathscr{L}_\pp^{\tt BDP}(f/K),
\end{equation}
where $\mathfrak{Z}^\lambda:=\mathfrak{Z}[1]^\lambda$. By the decomposition in Theorem~\ref{thm:Heeg}
and the construction of ${\rm Log}_{\pp,{\rm ac}}^\bullet$, the result follows.
\end{proof}

\subsection{Signed theta elements}\label{sec:ac}

Assume now that the pair $(f,K)$ is definite,
so that $N^-$ is the square-free product
of an \emph{odd} number of primes, and continue to assume that $p=\pp\overline{\pp}$ splits in $K$.

For each $n\geqslant 0$ let $\mathcal{G}_n:={\rm Gal}(K[p^n]/K)$ be the Galois group of
the ring class field of $K$ of conductor $p^n$; thus in particular,   $\mathcal{G}_0$ is the Galois group
of the Hilbert class field of $K$.

\begin{prop}\label{prop:theta}
There is a system of theta elements $\Theta_n\in\cO_L[\mathcal{G}_n]$
such that
\[
{\rm cor}^n_{n-1}(\Theta_n)=
\left\{
\begin{array}{ll}
a_p\cdot\Theta_{n-1}-\Theta_{n-2}&\textrm{if $n\geqslant 2$,}\\
\frac{1}{u_K}(a_p-\sigma_\pp-\sigma_{\overline\pp})\cdot\Theta_0&\textrm{if $n=1$,}
\end{array}
\right.
\]
where $u_K:=\vert\cO_K^\times\vert/2$, $\sigma_{\pp}, \sigma_{\overline{\pp}}\in\mathcal{G}_0$
are the Frobenius elements at $\pp$, $\overline{\pp}$, respectively, and
${\rm cor}^n_{n-1}:\cO_L[\mathcal{G}_n]\rightarrow\cO_L[\mathcal{G}_{n-1}]$ is the natural projection.
\end{prop}

\begin{proof}
See \cite[\S{2}]{BDmumford-tate}.
\end{proof}

Letting $\lambda\in\{\alpha,\beta\}$ be any of the roots of $X^2-a_pX+p$
(and after possibly enlarging $L$ so that it contains these roots),
the \emph{$\lambda$-regularized} theta-elements $\Theta_n^\lambda\in L[\mathcal{G}_n]$ defined by
\[
\Theta_{n}^\lambda:=
\left\{
\begin{array}{ll}
\Theta_n-\frac{1}{\lambda}\cdot\Theta_{n-1}&\textrm{if $n\geqslant 1$,}\\
\frac{1}{u_K}\left(1-\frac{1}{\lambda}\sigma_\pp\right)\left(1-\frac{1}{\lambda}\sigma_{\overline\pp}\right)
\cdot\Theta_0&\textrm{if $n=0$,}
\end{array}
\right.
\]
are immediately seen (in light of Proposition~\ref{prop:theta}) to satisfy the norm-compatibility
\[
{\rm cor}^n_{n-1}(\Theta_n^\lambda)=\lambda\cdot\Theta_{n-1}^\lambda
\]
for all $n>0$. Thus 
there exists a unique element
\[
\Theta_\infty^\lambda\in\mathcal{H}_{L,v_\lambda}(\Gamma^{\rm ac})
\] 
mapping to $\lambda^{-n}\cdot\Theta_n^\lambda$ under the natural projection
$\mathcal{H}_{L}(\Gamma^{\rm ac})\rightarrow L[\Gamma_n^{\rm ac}]$ for all $n\geqslant 0$ 
(see \cite[Prop.~2.8]{pollack} and the references therein, which readily adapt to our anticyclotomic setting).

Let $X_{N^+,N^-}$ be the Shimura set attached to a quaternion algebra $B/\bQ$ of discriminant $N^-$ 
and an Eichler order of level $N^+$.  
The module ${\rm Pic}(X_{N^+,N^-})$ is equipped with a natural action of  
the quotient of the Hecke algebra of level $N$ acting faithfully on subspace of $N^-$-new forms in $S_2(\Gamma_0(N))$,  
and in the construction of $\Theta_n$ (as explained in  \cite[\S{2.1}]{pollack-weston}, for example) 
one chooses a 
generator $\phi_f$ of its $f$-isotypical component. 

Taking $\phi_f$ to be normalized as in \cite[Lem.~2.1]{pollack-weston}, 
we define the \emph{Gross period} $\Omega_{f,N^-}$ by
\begin{equation}\label{def:Gross}
\Omega_{f,N^-}:=\frac{(4\pi)^2\langle f,f\rangle_N}{\langle\phi_f,\phi_f\rangle},
\end{equation}
where $\langle\;,\;\rangle$ is the intersection pairing on ${\rm Pic}(X_{N^+,N^-})$. 

\begin{prop}\label{prop:theta-bound}
For each $\lambda\in\{\alpha,\beta\}$, the elements $\Theta_\infty^\lambda$ satisfy the following interpolation property: 
If $\chi$ is a finite order character of $\Gamma^{\rm ac}$ of conductor $p^n>1$, then
\begin{equation}\label{eq:Gross}
\chi(\Theta_\infty^\lambda)^2=\frac{1}{\lambda^{2n}}\cdot\frac{L(f/K,\chi,1)}{\Omega_{f,N^-}}\cdot\sqrt{D}p^n\cdot u_K^2.
\end{equation}
Moreover, there exist elements 
$\Theta_\infty^{\sharp}, \Theta_\infty^{\flat}\in\Lambda_L(\Gamma^{\rm ac})$ such that
\[
\left(\begin{array}{cc}
\Theta_\infty^\alpha\\
\Theta_\infty^{\beta}
\end{array}\right)
=Q_{\alpha,\beta}^{-1}M_{{\rm log},{\rm ac}}\cdot
\left(\begin{array}{cc}
\Theta_\infty^{\sharp}\\
\Theta_\infty^{\flat}
\end{array}\right).
\]
\end{prop}

\begin{proof}
Note that by construction we have 
\[
\chi(\Theta_\infty^\lambda)^2=\lambda^{-2n}\cdot\chi(\Theta^\lambda_n)^2
=\lambda^{-2n}\cdot\chi(\Theta_n)^2,
\]
using that $\chi$ has conductor $p^n>1$ for the second equality. 
Letting $\psi_f:{\rm Pic}(X_{N^+,N^-})\rightarrow\cO_L$ denote the function given by 
$\psi_f(x)=\langle x,\phi_f\rangle$, the theta element $\Theta_n\in\cO_L[\mathcal{G}_n]$ is defined by
\[
\Theta_n:=\sum_{\sigma\in\mathcal{G}_n}\psi_f(P_n^\sigma)\sigma,
\]
where $P_n\in X_{N^+,N^-}$ is a certain ``Gross point'' of conductor $p^n$. As in  
\cite[Lem.~2.2]{pollack-weston} and \cite[Prop.~4.3]{ChHs1}, the interpolation property $(\ref{eq:Gross})$ 
then follows from Gross's special value formula; and with $(\ref{eq:Gross})$ in hand, the last claim in the proposition 
follows from Lemma~\ref{lem:coeffs}.
\end{proof}

The following application of Vatsal's results \cite{vatsal-special} will play an key role in our proof of Theorem~A in the definite setting.

\begin{prop}\label{prop:mu-def}
For each $\bullet\in\{\sharp,\flat\}$, we have $\mu(\Theta_\infty^\bullet)=0$.
\end{prop}

\begin{proof}
Following \cite{LLZ-Sha}, for each $n\geqslant 1$ we set
\[
\mathscr{H}_{n,{\rm ac}}:=\mathfrak{M}^{-1}((1+\pi)\varphi^{n-1}(P^{-1})\cdots\varphi(P^{-1}))
\in M_{2\times 2}(\mathcal{H}_L(\Gamma^{\rm ac})),
\]
where
\[
P=\left(\begin{array}{cc}0&-q^{-1}\\1 &a_p\delta q^{-1}\end{array}\right)
\in{\rm M}_{2\times 2}(\mathbb{B}_{{\rm rig},\bQ_p}^+)
\]
is the matrix of $\varphi$ on $\mathbb{N}(T_f^*)$
with respect to the basis $\{n_1,n_2\}$ in (\ref{eq:basis}), with $q=\varphi(\pi)/\pi\in\mathbb{A}_{\bQ_p}^+$
and $\delta\in(\mathbb{A}_{\bQ_p}^+)^\times$.  Letting $\gamma_{\rm ac}\in\Gamma^{\rm ac}$ be a fixed topological generator, by [\emph{loc.cit.}, Lem.~3.7] and the defining interpolating property of $\Theta_\infty^\lambda$ we have the congruence
\begin{equation}\label{eq:cong}
\left(\begin{array}{cc}
\alpha^{-n}\cdot\Theta_n^{\alpha}\\
\beta^{-n}\cdot\Theta_n^{\beta}
\end{array}\right)
\equiv Q_{\alpha,\beta}^{-1}\cdot A_\varphi^n\cdot\mathscr{H}_{n,{\rm ac}}\cdot
\left(\begin{array}{cc}
\Theta_n^{\sharp}\\
\Theta_n^{\flat}
\end{array}\right)\pmod{\gamma_{\rm ac}^{p^{n}}-1},\nonumber
\end{equation}
where $\Theta_n^\bullet$ is the image of $\Theta_\infty^\bullet$
under the natural projection $\Lambda(\Gamma^{\rm ac})\rightarrow\cO_L[\Gamma_n^{\rm ac}]$
and $A_\varphi$ is as in $(\ref{eq:A})$.
Since the matrix $Q_{\alpha,\beta}$ diagonalizes $A_\varphi$: 
\[
A_\varphi\cdot Q_{\alpha,\beta}=Q_{\alpha,\beta}\cdot\left(\begin{array}{cc}\alpha^{-1}&0\\0&\beta^{-1}\end{array}\right),
\]
using the definition of $\Theta_n^\lambda$ we see that $(\ref{eq:cong})$ reduces to 
\begin{equation}\label{eq:cong-bis}
(\alpha-\beta)\left(\begin{array}{cc}
\Theta_n\\
\Theta_{n-1}
\end{array}\right)
\equiv\mathscr{H}_{n,{\rm ac}}\cdot
\left(\begin{array}{cc}
\Theta_n^{\sharp}\\
\Theta_n^{\flat}
\end{array}\right)\pmod{\gamma_{\ac}^{p^{n}}-1}.\nonumber
\end{equation}

Now let $\zeta_{p^n}$ be a primitive $p^n$-th root of unity, and set $\epsilon_n:=\zeta_{p^n}-1$.
Viewing $\Theta_\infty^\bullet$ as elements in $\cO_L[[T]]$, let $\lambda^\bullet$ and $\mu^\bullet$
be the corresponding $\lambda$- and $\mu$-invariants, and set $\lambda:={\rm min}(\lambda^\sharp,\lambda^\flat)$.
Then from $(\ref{eq:cong-bis})$
and \cite[Cor.~4.7]{LLZ-Sha} we obtain the estimates
\begin{equation}\label{eq:estimates}
{\rm ord}_p(\Theta_n(\epsilon_n))\geqslant
\left\{
\begin{array}{ll}
\lambda+(p^n-p^{n-1})\bigl(\frac{\mu^\bullet}{e}+\sum_{i=1}^{\frac{n-1}{2}}p^{1-2i}\bigr)&\textrm{if $n$ is odd},\\
\lambda+(p^n-p^{n-1})\bigl(\frac{\mu^\bullet}{e}+\sum_{i=1}^{\frac{n}{2}}p^{-2i}\bigr)&\textrm{if $n$ is even},
\end{array}
\right.
\end{equation}
for each of the signs $\bullet\in\{\sharp,\flat\}$, where $e$ is the absolute ramification degree of $L$.
Since on the other hand 
the arguments in \cite[\S{5.9}]{vatsal-special} imply that $\mu(\Theta_n)=0$ for $n$ sufficiently large
(see also \cite[Thm.~2.5]{pollack-weston}), we conclude from (\ref{eq:estimates}) and \cite[Lem.~5.5]{PR-exp}
that $\mu^\bullet=0$, as was to be shown.
\end{proof}

\section{Main results}
In this section we conclude the proof of the results stated in the Introduction. The proof of Theorem~B naturally splits into two cases according to the sign $\epsilon(\pi,K,\frac{1}{2})$ (see the discussion at the beginning of $\S\ref{sec:ac-Iw}$), but in both cases it builds on the following result toward the Iwasawa--Greenberg main conjecture for the $p$-adic $L$-function $\mathscr{L}_\pp(f/K)$ 
appearing in Theorem~\ref{thm:ERL-pp}:

\begin{thm}\label{thm:wan}
	Assume that: 
	\begin{itemize}
		\item{} $N$ is square-free,
		\item{} $\bar{\rho}_f$ is ramified at some prime at every prime $q\mid N$ which is nonsplit in $K$, and there is at least one such prime,
		\item{} $\bar{\rho}_f\vert_{G_K}$ is irreducible.
	\end{itemize}
	Then we have the divisibility
	\[
	{\rm Char}_{\Lambda(\Gamma_K)}(\mathfrak{X}^{{\rm rel},{\rm str}}_{K_\infty}(f))
	\subseteq(\mathscr{L}_\pp(f/K))
	\]
	as ideals in $\Lambda_{R_0}(\Gamma_K)\otimes_{\bZ_p}\bQ_p$, up to primes which are pullbacks for height one primes of $\Lambda(\Gamma_K^+)$.
\end{thm}

\begin{proof}
	See part (2) of \cite[Thm.~5.3]{wan-combined}.
\end{proof}

The following result will be repeatedly used to descend to the cyclotomic line.

\begin{prop}\label{prop:control}
	For each $\bullet\in\{\sharp,\flat\}$ there are $\Lambda(\Gamma^{\rm cyc})$-module isomorphisms
	\begin{align*}
	X_{K_\infty}^{\bullet,\bullet}(f)/(\gamma_{\rm ac}-1)X_{K_\infty}^{\bullet,\bullet}(f)
	&\simeq X_{K_\infty^{\rm cyc}}^{\bullet,\bullet}(f)\\
	&\simeq X_{\bQ_\infty}^\bullet(f)\oplus
	X_{\bQ_\infty}^\bullet(f_K).
	\end{align*}
	In particular, letting $\mathcal{L}_p^{\bullet,\bullet}(f/K)_{\rm cyc}$ be the image of a generator of ${\rm Char}_{\Lambda(\Gamma_K)}(X_{K_\infty}^{\bullet,\bullet}(f))$ under the natural projection $\Lambda(\Gamma_K)\twoheadrightarrow\Lambda(\Gamma^{\rm cyc})$, we have the divisibility
	\[
	(\mathcal{L}_p^{\bullet,\bullet}(f/K)_{\rm cyc})\supseteq
	{\rm Char}_{\Lambda(\Gamma^{\rm cyc})}(X_{\bQ_\infty}^\bullet(f))\cdot
	{\rm Char}_{\Lambda(\Gamma^{\rm cyc})}(X_{\bQ_\infty}^\bullet(f_K)).
	\]
\end{prop}

\begin{proof}
	This follows from a straightforward extension of the arguments in Proposition~3.9 and Lemma~3.6 of \cite{SU}.	
\end{proof}

\subsection{Indefinite case}\label{subsec:indef}

In this section, we finish the proof of Theorem~B in the case where $N^-$ is 
the product of an \emph{even} number of primes. This extends to general non-ordinary modular forms of weight $2$ the main results of
\cite{cas-wan-ss} 
on B.-D.~Kim's doubly-signed main conjectures, 
which assumed $a_p=0$.

A key ingredient in the proof will be the relation between Conjecture~\ref{conj:HP} and Conjecture~\ref{conj:ac-MC} below, corresponding to the analogue of Perrin-Riou's main conjecture \cite{PR-HP} for the signed Heegner points constructed in $\S\ref{sec:HP}$, and Greenberg's main conjecture \cite{Greenberg55} for $\mathscr{L}^{\tt BDP}_\pp(f/K)$, respectively.

Recall that attached to $f$ is an isogeny class of abelian variety quotients
\begin{equation}\label{eq:opt}
\pi_f:J(X)\longrightarrow A_f,
\end{equation}
where $J(X)$ is the Jacobian variety of a Shimura curve $X_{N^+,N^-}$ attached to an indefinite rational quaternion algebra of discriminant $N^-$. We assume that $(\ref{eq:opt})$ is an optimal quotient, i.e., ${\rm ker}(\pi_f)$ is connected, and set $\delta(N^+,N^-)=\pi_f\circ\pi_f^\vee$, where $\phi_f^\vee$
is the dual isogeny. Replacing $(\ref{eq:opt})$ by a classical modular  parametrization
\[
\varphi_f:J_0(N)\longrightarrow A_f
\]
we similarly define $\delta(N):=\delta(N,1)$, and set
\[
\delta_{N^+,N^-}:=\frac{\delta(N^+,N^-)}{\delta(N)}.
\]
Finally, let $c_f$ be the Manin constant associated with $\varphi_f$ (see e.g \cite{ARS}) and $u_K:=\vert\cO_K^\times\vert/2$.

\begin{conj}\label{conj:HP}
For each $\bullet\in\{\sharp,\flat\}$, both $\mathfrak{X}_{K_\infty^{\ac}}^{\bullet,\bullet}(f)$
and $\mathfrak{Sel}^{\bullet,\bullet}(K,\mathbf{T}^{\rm ac})$ have $\Lambda(\Gamma^{\rm ac})$-rank $1$, and
\[
c^2\cdot{\rm Char}_{\Lambda(\Gamma^{\rm ac})}(\mathfrak{X}_{K_\infty^{\ac}}^{\bullet,\bullet}(f)_{\rm tors})
=\frac{\delta_{N^+,N^-}}{c_f^2u_K^2}\cdot{\rm Char}_{\Lambda(\Gamma^{\rm ac})}\biggl(\frac{\mathfrak{Sel}^{\bullet,\bullet}(K,\mathbf{T}^{\rm ac})}
{\Lambda(\Gamma^{\rm ac})\cdot\mathfrak{Z}_c^\bullet}\biggr)^2
\]
as ideals in $\Lambda(\Gamma^{\rm ac})$, where the subscript {\rm tors} denotes the $\Lambda(\Gamma^{\rm ac})$-torsion submodule.
\end{conj}

\begin{conj}\label{conj:ac-MC}
The module $X_{K_\infty^{\ac}}^{{\rm rel},{\rm str}}(f)$ is $\Lambda(\Gamma^{\rm ac})$-torsion, and
\[
{\rm Char}_{\Lambda(\Gamma^{\rm ac})}(X_{K_\infty^{\ac}}^{{\rm rel},{\rm str}}(f))
=(\mathscr{L}_\pp^{\tt BDP}(f/K)^2)
\]
as ideals in $\Lambda_{\unr}(\Gamma^{\rm ac})$.
\end{conj}

We record the following auxiliary results for our later use.

\begin{lem}\label{lem:prelim}
Let $\bullet\in\{\sharp,\flat\}$, and assume that $\mathfrak{Sel}^{\bullet,\bullet}(K,\mathbf{T}^{\rm ac})$
has $\Lambda(\Gamma^{\rm ac})$-rank $1$. Then $\mathfrak{X}_{K_\infty^{\ac}}^{\bullet,\bullet}(f)$ has $\Lambda(\Gamma^{\rm ac})$-rank $1$, and the following statements hold:
\begin{enumerate}
\item{} $\mathfrak{X}_{K_\infty^{\ac}}^{\bullet,{\rm str}}(f)$ is $\Lambda(\Gamma^{\rm ac})$-torsion and the inclusion
\[
\mathfrak{Sel}^{\bullet,\bullet}(K,\mathbf{T}^{\rm ac})\subseteq\mathfrak{Sel}^{\bullet,{\rm rel}}(K,\mathbf{T}^{\rm ac})
\]
is an equality.
\item{} $\mathfrak{X}_{K_\infty^{\ac}}^{{\rm rel},{\rm str}}(f)$ is $\Lambda(\Gamma^{\rm ac})$-torsion, and for
any height one prime 
of $\Lambda(\Gamma^{\rm ac})$, we have 
\begin{align*}
{\rm length}_{\mathfrak{P}}(\mathfrak{X}_{K_\infty^{\ac}}^{{\rm rel},{\rm str}}(f))
&={\rm length}_{\mathfrak{P}}(\mathfrak{X}_{K_\infty^{\ac}}^{\bullet,\bullet}(f)_{\rm tors})
+2\;{\rm length}_{\mathfrak{P}}({\rm coker}({\rm res}_\pp))
\end{align*}
and
\begin{align*}
{\rm ord}_{\mathfrak{P}}(c\cdot\mathscr{L}^{\tt BDP}_\pp(f/K))&={\rm length}_{\mathfrak{P}}({\rm coker}({\rm res}_\pp))
+{\rm length}_{\mathfrak{P}}\left(\frac{\mathfrak{Sel}^{\bullet,\bullet}(K,\mathbf{T}^{\rm ac})}
{\Lambda(\Gamma^{\rm ac})\cdot\mathfrak{Z}_c^\bullet}\right),
\end{align*}
where ${\rm res}_\pp:\mathfrak{Sel}^{\bullet,\bullet}(K,\mathbf{T}^{\rm ac})\rightarrow H_\bullet^1(K_\pp,\mathbf{T}^{\rm ac})$
is the natural restriction map.
\end{enumerate}
\end{lem}

\begin{proof}
The proof of the corresponding three lemmas in \cite[\S{4.3}]{cas-wan-ss} applies almost verbatim,
using the explicit reciprocity law of Theorem~\ref{thm:ERL-bdp} in place of
[\emph{loc.cit.}, Thm.~4.6] for the proof of the second equality in part (2).
\end{proof}

Fix a finite set of places $\Sigma$ of $K$ containing those dividing $Np\infty$. For each $\bullet\in\{\sharp,\flat\}$, define the Selmer structure
$\mathcal{F}^\bullet$ on $\mathbf{T}^{\rm ac}$ (in the sense of \cite[\S{1.1}]{howard-PhD-I})
by taking the unramified local condition
\[
H^1_{\rm ur}(K_v,\mathbf{T}^{\rm ac}):={\rm ker}\left\{H^1(K_v,\mathbf{T}^{\rm ac})\longrightarrow
H^1(K_v^{\rm ur},\mathbf{T}^{\rm ac})\right\}
\]
at the places in $v\in\Sigma$ not dividing $p$, and the
local condition $H^1_{\bullet}(K_v,\mathbf{T}^{\rm ac})\subset H^1(K_v,\mathbf{T}^{\rm ac})$
at the primes $v$ dividing $p$. (Thus the Selmer group denoted $\mathfrak{Sel}^{\bullet,\bullet}(K,\mathbf{T}^{\rm ac})$ in $\S\ref{sec:Sel}$ corresponds to $H^1_{\mathcal{F}^\bullet}(K,\mathbf{T}^{\rm ac})$ in the notations of \cite{howard-PhD-I}.)

For the statement of the next result, we refer the reader to \cite[\S{1.2}]{howard-PhD-I} for
the definition of the module of Kolyvagin systems $\mathbf{KS}(\mathbf{T}^{\rm ac},\mathcal{F},\mathcal{L})$
attached to a Selmer structure $\mathcal{F}$ on $\mathbf{T}^{\rm ac}$ and a certain set $\mathcal{L}$
of primes inert in $K$.

\begin{prop}\label{prop:KS}
For each $\bullet\in\{\sharp,\flat\}$, there is a Kolyvagin system
$\kappa^\bullet\in\mathbf{KS}(\mathbf{T}^{\rm ac},\mathcal{F}^\bullet,\mathcal{L})$
with $\kappa_1^\bullet=\mathfrak{Z}_c^\bullet$.
\end{prop}

\begin{proof}
As in the proof of the corresponding result in \cite[Thm.~4.14]{cas-wan-ss}
(to which we refer the reader for the definition of the ideal $I_m\subset p\cO_L$, noting that $S$ is \emph{loc.cit.}
corresponds to $m$ here), this is reduced to showing that the classes
\[
\kappa_m^\bullet\in H^1(K,\mathbf{T}^{\rm ac}/I_m\mathbf{T}^{\rm ac}),
\]
obtained by applying the `Kolyvagin's derivative' construction in \cite[\S{1.7}]{howard-PhD-I}
to the classes $\mathfrak{Z}[m]^\bullet$ of Theorem~\ref{thm:Heeg}, 
are such that
\begin{equation}\label{eq:key-incl}
{\rm res}_v(\kappa_m^\bullet)\in
{\rm im}\left(H^1_{\bullet}(K_v,\mathbf{T}^{\rm ac})\longrightarrow H^1(K_v,\mathbf{T}^{\rm ac}/I_m\mathbf{T}^{\rm ac})
\right)
\end{equation}
for all $v\mid p$. To show this, let $G(m)=\prod_{\ell}G(\ell)$ be the Galois group of the extension $K[m]/K[1]$, and recall that Kolyvagin's derivative operator $D_m=\prod_{\ell}D_\ell\in\bZ[G(m)]$ is defined by $D_\ell=\sum_{i=1}^\ell i\sigma_\ell^i$, where $\sigma_\ell\in G(\ell)$ is any generator. Setting 
\begin{equation}\label{eq:kol-der-op}
Z[m]^\bullet:=\sum_\sigma\sigma D_m\mathfrak{Z}_c[m]\in H^1(K[m],\Tc)_,
\end{equation}
where the sum is over a complete set of representatives of the quotient of $\mathcal{G}(m):={\rm Gal}(K[m]/K)$ by $G(m)$, one readily checks that the natural image $\bar{Z}[m]^\bullet$ of $Z[m]^\bullet$ in $H^1(K[m],\Tc/I_m\Tc)$ is fixed by $\mathcal{G}(m)$; the derivative class $\kappa_m^\bullet$ is then determined by the condition that
\begin{equation}\label{eq:kol-der}
{\rm res}(\kappa_m^\bullet)=\bar{Z}[m]^\bullet
\end{equation}
under the isomorphism ${\rm res}:H^1(K,\Tc/I_m\Tc)\simeq H^1(K[m],\Tc/I_m\Tc)^{\mathcal{G}(m)}$. 

Now, if $w$ is any place of $K[m]$ above $v$, by Lemma~\ref{lem:Heeg-Col} we have the inclusion 
\[
{\rm res}_w(\mathfrak{Z}_c[m]^\bullet)\in H^1_\bullet(K[m]_w,\Tc),
\] 
and hence from $(\ref{eq:kol-der-op})$ and $(\ref{eq:kol-der})$, the inclusion $(\ref{eq:key-incl})$ follows.
\end{proof}

Equipped with the Kolyvagin system of Proposition~\ref{prop:KS} (whose
nontriviality is guaranteed by Theorem~\ref{thm:ERL-bdp})
the following result towards Conjecture~\ref{conj:HP} follows easily.

\begin{thm}\label{thm:KS-argument}
Let $\bullet\in\{\sharp,\flat\}$, and assume that $\bar{\rho}_f\vert_{G_K}$ is irreducible. 
Then both $\mathfrak{X}_{K_\infty^{\ac}}^{\bullet,\bullet}(f)$
and $\mathfrak{Sel}^{\bullet,\bullet}(K,\mathbf{T}^{\rm ac})$ have $\Lambda(\Gamma^{\rm ac})$-rank $1$, and
we have the divisibility
\[
c^2\cdot{\rm Char}_{\Lambda(\Gamma^{\rm ac})}(\mathfrak{X}_{K_\infty^{\ac}}^{\bullet,\bullet}(f)_{\rm tors})
\supseteq{\rm Char}_{\Lambda(\Gamma^{\rm ac})}\left(\frac{\mathfrak{Sel}^{\bullet,\bullet}(K,\mathbf{T}^{\rm ac})}
{\Lambda(\Gamma^{\rm ac})\cdot\mathfrak{Z}_c^\bullet}\right)^2.
\]
\end{thm}

\begin{proof}
This follows from a straightforward adaptation of
the arguments in the last paragraph of \cite[\S{4.3}]{cas-wan-ss},
with the self-duality of the
plus/minus local conditions quoted from \cite{kim-parity} in \emph{loc.cit.}
replaced by the corresponding results for the signed local conditions of $\S\ref{sec:Sel}$ above (see e.g. Lemma~2.4 and Remark~2.5 in \cite{hatley-lei}).
\end{proof}

Now we assemble all the pieces, yielding the follow result on Conjecture~\ref{conj:ac-MC}. 

\begin{thm}\label{thm:ac-MC}
	Let $f\in S_2(\Gamma_0(N))$ be a newform, let $p\nmid N$ be an odd prime,
	and let $K/\bQ$ be an imaginary quadratic field in which $p=\pp\overline{\pp}$ splits.
	Assume that $f$ is non-ordinary at $p$, and that
	\begin{itemize}
		\item{} $N$ is square-free,
		\item{} $N^-$ is divisible by an even number of primes,
		\item{} $\bar{\rho}_f$ is ramified at every prime $\ell\mid N$ which is nonsplit in $K$, and there is at least one such prime,
		\item{} $\bar{\rho}_f\vert_{G_K}$ is irreducible.
	\end{itemize}
	Then $X_{K_\infty^{\ac}}^{{\rm rel},{\rm str}}(f)$ is $\Lambda(\Gamma^{\ac})$-cotorsion, and
	\[
	Ch_{\Lambda(\Gamma^{\rm ac})}(X_{K_\infty^{\ac}}^{{\rm rel},{\rm str}}(f))=(\mathscr{L}^{\tt BDP}_\pp(f/K)^2)
	\]
	as ideals in $\Lambda_{\unr}(\Gamma^{\rm ac})$. 
	That is, under the stated hypotheses Conjecture~\ref{conj:ac-MC} holds.
\end{thm}

\begin{proof}
We shall adapt the arguments in \cite{cas-wan-ss}, indicating the necessary adjustments. From Theorem~\ref{thm:KS-argument} and Lemma~\ref{lem:prelim}, we know that $\mathfrak{Sel}^{\bullet,\bullet}(K,\mathbf{T}^{\ac})$ and
$\mathfrak{X}_{K_\infty^{\ac}}^{\bullet,\bullet}(f)$ both have $\Lambda(\Gamma^{\ac})$-rank $1$, $\mathfrak{X}_{K_\infty^{\ac}}^{{\rm rel},{\rm str}}(f)$ is $\Lambda(\Gamma^{\ac})$-torsion,
and for any height one prime $\mathfrak{P}$ of $\Lambda(\Gamma^{\ac})$, the inequalities
\begin{equation}\label{eq:1}
{\rm length}_{\mathfrak{P}}(c^2\cdot\mathfrak{X}_{K_\infty^{\ac}}^{\bullet,\bullet}(f)_{\rm tors})
\leqslant2\;{\rm length}_{\mathfrak{P}}
\left(\frac{\mathfrak{Sel}^{\bullet,\bullet}(K,\mathbf{T}^{\rm ac})}{\Lambda(\Gamma^{\rm ac})\cdot\mathfrak{Z}_c^\bullet}\right)
\end{equation}
and
\begin{equation}\label{eq:2}
{\rm length}_{\mathfrak{P}}(\mathfrak{X}_{K_\infty^{\ac}}^{{\rm rel},{\rm str}}(f))
\leqslant 2\;{\rm length}_{\mathfrak{P}}(\mathscr{L}^{\tt BDP}_\pp(f/K))
\end{equation}
hold. On the other hand, the divisibility in Theorem~\ref{thm:wan} combined with
the comparison of $p$-adic $L$-functions in \cite[Cor.~1.12]{cas-wan-ss} and a standard control theorem (as in \cite[Prop.~3.9]{SU} and \cite[\S{3.2}]{wan}) implies that
\begin{equation}\label{eq:wan-div}
{\rm Char}_{\Lambda(\Gamma^{\rm ac})}(\mathfrak{X}_{K_\infty^{\ac}}^{{\rm rel},{\rm str}}(f))
\subseteq(\mathscr{L}^{\tt BDP}_\pp(f/K)^2)
\end{equation}
as ideals in $\Lambda_{\unr}(\Gamma^{\rm ac})$, using the fact that by the vanishing of the $\mu$-invariant of $\mathscr{L}^{\tt BDP}_\pp(f/K)$ (see Theorem~\ref{thm:bdp}), the ambiguity by powers of $p$ and pullbacks of height one primes of $\Lambda(\Gamma_K^+)$ in Theorem~\ref{thm:wan} can be removed. Now, $(\ref{eq:wan-div})$ combined with $(\ref{eq:2})$ yields the  equality in Conjecture~\ref{conj:ac-MC} with $\mathfrak{X}_{K_\infty^{\ac}}^{{\rm rel},{\rm str}}(f)$ in place of $X_{K_\infty^{\ac}}^{{\rm rel},{\rm str}}(f)$; since by \cite[\S{3}]{pollack-weston} (see also \cite[Prop.~2.5]{cas-mult}) both Selmer modules have the same characteristic ideal under our ramification hypotheses on $\bar\rho_f$, the result follows.
\end{proof}

In particular, Theorem~\ref{thm:ac-MC} yields the divisibility opposite to Theorem~\ref{thm:wan} on the Iwasawa--Greenberg main conjecture for $\mathscr{L}_\pp(f/K)$: 

\begin{cor}\label{thm:2var-IwGr}
Under the hypotheses of Theorem~\ref{thm:ac-MC}, the module $\mathfrak{X}_{K_\infty}^{{\rm rel},{\rm str}}(f)$ is $\Lambda(\Gamma_K)$-torsion, and we have
\[
{\rm Char}_{\Lambda(\Gamma_K)}(X_{K_\infty}^{{\rm rel},{\rm str}}(f))=(\mathscr{L}_\pp(f/K))
\]
as ideals in $\Lambda_{\unr}(\Gamma_K)$, 
\end{cor}

\begin{proof}
This follows by combining Theorem~\ref{thm:ac-MC} and the divisibility in Theorem~\ref{thm:wan} as in \cite[Thm.~5.2]{cas-wan-ss}.
\end{proof}

We can now conclude the proof of our first main result.

\begin{proof}[Proof of Theorem~B in the indefinite case]
Let $\bullet\in\{\sharp,\flat\}$ be such that 
$\mathfrak{L}_p^{\bullet,\bullet}(f/K)_{\rm cyc}$ is nonzero, as assumed in the statement.   
By Theorem~\ref{thm:equiv}, the conclusion of Corollary~\ref{thm:2var-IwGr} then implies that  $\mathfrak{X}_{K_\infty}^{\bullet,\bullet}(f)$ is $\Lambda(\Gamma_K)$-torsion and we have the equality
\begin{equation}\label{eq:div-exc}
c\cdot{\rm Char}_{\Lambda(\Gamma_K)}(\mathfrak{X}^{\bullet,\bullet}_{K_\infty}(f))=(\mathfrak{L}_p^{\bullet,\bullet}(f/K))
\end{equation}
as ideals in $\Lambda(\Gamma_K)$, up to exceptional primes. Thus to conclude the proof it suffices to show that neither side in equality $(\ref{eq:div-exc})$ is divisible by exceptional primes. For the right-hand side, this follows from the nonvanishing of    $\mathfrak{L}_p^{\bullet,\bullet}(f/K)_{\rm cyc}$; for the left-hand side, by Proposition~\ref{prop:cyc-comp} the same nonvanishing implies that both $L_p^\bullet(f)$ and $L_p^\bullet(f_K)$ are nonzero, and so both $X_{\bQ_\infty}^\bullet(f)$ and $X_{\bQ_\infty}^\bullet(f_K)$ are $\Lambda(\Gamma^{\rm cyc})$-cotorsion by \cite[Thm.~6.5]{LLZ-AJM} (a consequence of Kato's work \cite{Kato295}). Combined with Proposition~\ref{prop:control}, this shows that ${\rm Char}_{\Lambda(\Gamma_K)}(\mathfrak{X}_{K_\infty}^{\bullet,\bullet}(f))$ has no exceptional prime divisors, concluding the proof of Theorem~B
in the indefinite case.
\end{proof}

\subsection{Definite case}\label{subsec:def}

In this section we complete the proof of Theorem~B in the case in which $N^-$ has an \emph{odd} number of prime factors, and give the proof of Theorem~A, our main result on the signed main conjectures of Lei--Loeffler--Zerbes \cite{LLZ-AJM}. 

Key to both proofs will be the following intermediate result:

\begin{thm}\label{thm:2var-div}
	Assume that:
	\begin{itemize}
		\item{} $N$ is square-free,
		\item{} $N^-$ has an odd number of prime factors,
		\item{} $\bar{\rho}_f$ is ramified at every prime $\ell\mid N$ which is nonsplit in $K$,
		\item{} $\bar{\rho}_f\vert_{G_K}$ is irreducible.
	\end{itemize}
	If $\bullet,\circ\in\{\sharp,\flat\}$ are such that the restriction $\mathfrak{L}_p^{\bullet,\circ}(f/K)_{\rm cyc}$ is nonzero, then the module $\mathfrak{X}_{K_\infty}^{\bullet,\circ}(f)$ is $\Lambda(\Gamma_K)$-torsion, and we have the divisibility
	\[
	c\cdot{\rm Char}_{\Lambda(\Gamma_K)}(\mathfrak{X}^{\bullet,\circ}_{K_\infty}(f))
	\subseteq(\mathfrak{L}_p^{\bullet,\circ}(f/K))
	\]
	as ideals in $\Lambda(\Gamma_K)$.		
\end{thm}

\begin{proof}
The arguments in the proof of Theorem~\ref{thm:equiv} show that Theorem~\ref{thm:wan} and the nonvanishing of $\mathfrak{L}_p^{\bullet,\circ}(f/K)$ implies that  $\mathfrak{X}_{K_\infty}^{\bullet,\circ}(f)$ is $\Lambda(\Gamma_K)$-torsion, and that the following divisibility holds:
\[
c\cdot{\rm Char}_{\Lambda(\Gamma_K)}(\mathfrak{X}^{\bullet,\circ}_{K_\infty}(f))
\subseteq(\mathfrak{L}_p^{\bullet,\circ}(f/K))
\]
as ideals in $\Lambda(\Gamma_K)$, up to powers of $P$ and primes which are pullback of height one primes of $\Lambda(\Gamma_K^+)$,  indeterminacies that can be removed thanks to the nonvanishing of the restriction $\mathfrak{L}_p^{\bullet,\circ}(f/K)_{\rm cyc}$ and  Proposition~\ref{prop:mu-def}, respectively.
\end{proof}

The conclusion of the proof of Theorem~B in the definite case will build on Theorem~A, so we begin by proving the latter.

\begin{proof}[Proof of Theorem~A] 
We shall apply Theorem~\ref{thm:2var-div} for a suitable choice of $K$, and descend to the cyclotomic line. 
Indeed, Ribet's level-lowering result \cite[Thm.~1.1]{ribet-eps} forces the residual representation $\bar{\rho}_f$ to be ramified at at least one prime $q$, which we fix. Let $K$ be an imaginary quadratic field such that:
\begin{itemize}
	\item[(a)] $q$ is inert in $K$,
	\item[(b)] every prime dividing $N/q$ splits in $K$,
	\item[(c)] $p$ splits in $K$,
	\item[(d)] $L(f_K,1)\neq 0$.
\end{itemize}
(The existence of such $K$ is guaranteed by \cite[Thm.~B]{FH}.)
By \cite{Edi}, the residual representation $\bar{\rho}_f$ is irreducible, and by \cite[Lem.~2.8.1]{skinner} (using that $\bar\rho_f$ is ramified at the prime $q$), the restriction $\bar{\rho}_f\vert_{G_K}$ remains irreducible. Thus the triple $(f,K,p)$ satisfies the conditions of Theorem~\ref{thm:2var-div}, whose divisibility combined with Proposition~\ref{prop:cyc-comp} and Proposition~\ref{prop:control} yields the divisibility
\begin{equation}\label{eq:props}
(L_p^\bullet(f)\cdot L_p^\bullet(f_K))\supseteq {\rm Char}_{\Lambda(\Gamma^{\rm cyc})}(X_{\bQ_\infty}^\bullet(f))
\cdot {\rm Char}_{\Lambda(\Gamma^{\rm cyc})}(X_{\bQ_\infty}^\bullet(f_K)),
\end{equation} 
where we used the fact that we may replace  $X_{K_\infty}^{\bullet,\bullet}(f)$ by $\mathfrak{X}_{K_\infty}^{\bullet,\bullet}(f)$ in Proposition~\ref{prop:control} (since $\bar{\rho}_f$ is ramified at every prime dividing $N^-=q$). 

On the other hand, the function $L_p^\bullet(f)$ is nonzero by hypothesis, while the nonvanishing of $L_p^\bullet(f_K)$ follows from condition (d) above and the interpolation property of $L_p^\bullet(f_K)$ at the trivial character (see \cite[Prop.~3.28]{LLZ-AJM}). Hence by \cite[Thm.~6.5]{LLZ-AJM}\footnote{Note that Assumption~(A) in \emph{loc.cit.} is only used to guarantee the nonvanishing of the $p$-adic $L$-function.} we have the divisibilities
\[
(L_p^\bullet(f))\subseteq {\rm Char}_{\Lambda(\Gamma^{\rm cyc})}(X^\bullet_{\bQ_\infty}(f))
\quad\textrm{and}\quad
(L_p^\bullet(f_K))\subseteq {\rm Char}_{\Lambda(\Gamma^{\rm cyc})}(X^\bullet_{\bQ_\infty}(f_K)).
\]
Since having strict inclusion in either of these would contradict (\ref{eq:props}), the proof of Theorem~A follows. In particular, in light of \cite[Cor.~6.6]{LLZ-AJM} Kato's main conjecture \cite[Conj.~12.5]{Kato295} holds. 
\end{proof}

\begin{proof}[Proof of Theorem~B in the definite case] 
This now follows immediately from the combination of Theorem~\ref{thm:2var-div} and Theorem~A. Indeed, let $\bullet\in\{\sharp,\flat\}$ be such that $\mathfrak{L}_p^{\bullet,\bullet}(f/K)_{\rm cyc}$ is nonzero,  
and define the ideals $X,  Y\subseteq\Lambda(\Gamma_K)$ by
\[
X:=c\cdot{\rm Char}_{\Lambda(\Gamma_K)}(\mathfrak{X}_{K_\infty}^{\bullet,\bullet}(f)),\quad
Y:=(\mathfrak{L}_p^{\bullet,\bullet}(f/K)),
\]	
and let $I^{\rm ac}=(\gamma_{\ac}-1)$ be the kernel of the projection $\Lambda(\Gamma_K)\twoheadrightarrow\Lambda(\Gamma^{\rm cyc})$. By Theorem~\ref{thm:2var-div} we have the divisibility $X\subseteq Y$, which combined with Proposition~\ref{prop:control} yields the divisibilities
\begin{align*}
c\cdot{\rm Char}_{\Lambda(\Gamma^{\rm cyc})}(X_{\bQ_\infty}^\bullet(f))
\cdot {\rm Char}_{\Lambda(\Gamma^{\rm cyc})}(X_{\bQ_\infty}^\bullet(f_K))
&\subseteq (X\;{\rm mod}\;I^{\rm ac})\\
&\subseteq (Y\;{\rm mod}\;I^{\rm ac})\\
&=(c\cdot L_p^\bullet(f)\cdot L_p^\bullet(f_K)),
\end{align*}
where the last equality is given by Proposition~\ref{prop:cyc-comp}. 
Since the extremes of this string are equal by Theorem~A applied to $f$ and $f_K$, we conclude that $(X\;{\rm mod}\;I^{\rm ac})=(Y\;{\rm mod}\;I^{\rm ac})$, and hence $X=Y$ by \cite[Lem.~3.2]{SU}, finishing the proof of Theorem~B. 
\end{proof}

\subsection{BSD formulae}\label{subsec:BSD}

In this section we deduce the applications of the previous results to the $p$-part of the Birch--Swinnerton-Dyer formula 
for abelian varieties over $\bQ$ of ${\rm GL}_2$-type.

\begin{proof}[Proof of Theorem~C]
Let $A/\bQ$ be an abelian variety of ${\rm GL}_2$-type as in the statement of Theorem~C. By \cite[Cor.~10.2]{KW-SerreI}, the is a newform $f\in S_2(\Gamma_0(N))$ such that
\[
L(A,s)=\prod _\sigma L(f^\sigma,s), 
\]
where $f^\sigma$ runs over the Galois conjugates of $f$. Let $L$ be the completion of the Hecke field of $f$ at the prime $\mathfrak{P}$ above $p$ inducted by our fixed isomorphism $\bC\simeq\bC_p$; then  $T_f^*:=T_\mathfrak{P}A$, the $\mathfrak{P}$-adic Tate module of $A$ and the residual representation $\bar\rho_f$ is realized in the $\mathfrak{P}$-torsion of $A$. Take a sign $\bullet\in\{\sharp,\flat\}$, and define $H^1_\bullet(\bQ_p,A[\mathfrak{P}^\infty])$ to be the image of $H^1_\bullet(\bQ,\mathbf{T}^{\rm cyc})$ under the natural map 
\[
H^1_\bullet(\bQ_p,\mathbf{T}^{\rm cyc})\longrightarrow 
H^1(\bQ_p,T_f^*)\longrightarrow H^1(\bQ_p,A[\mathfrak{P}^\infty]),
\]
and similarly define the unramified local condition $H^1_{\rm ur}(\bQ_\ell,A[\mathfrak{P}^\infty])$ for primes $\ell\neq p$. Letting ${\rm Sel}^\bullet_{\bQ}(f)\subseteq H^1(\bQ,A[\mathfrak{P}^\infty])$ be defined by the same recipe as ${\rm Sel}^\bullet(\bQ,\mathbf{A}^{\rm cyc})$ (see Remark~\ref{rem:Sel-Q}), we then have 
\[
{\rm Sel}^\bullet_{\bQ}(f)={\rm Sel}_{\mathfrak{P}^\infty}(A/\bQ) 
\]
(see \cite[Prop.~2.14]{hatley-lei}). Since by \cite[Thm.~1.1]{ribet-eps} 
the residual representation $\bar{\rho}_f$ is ramified at some prime $q$, the result then follows from Theorem~A, the interpolation property satisfied by $L_p^\bullet(f)$ at the trivial character 
(see \cite[Prop.~3.28]{LLZ-AJM}), and a variant of \cite[Thm.~4.1]{greenberg-cetraro} for signed Selmer groups (see \cite[\S{4.1}]{sprung-IMC}).	
\end{proof}

\begin{proof}[Proof of Theorem~D]
A straightforward adaptation of the arguments in \cite[\S{7.4}]{JSW}; in fact, in light of Theorem~\ref{thm:ac-MC}, it suffices to adapt the slightly simpler argument in \cite[\S{5}]{cas-mult}. We briefly recall the details for the convenience of the reader. 

Let $A/\bQ$ be an abelian variety of ${\rm GL}_2$-type as in the statement of Theorem~D. Then, just as in the proof of Theorem~C, we have 
\begin{equation}\label{eq:ES}
L(A,s)=\prod_{\sigma} L(f^\sigma,s)
\end{equation}
for some newform $f\in S_2(\Gamma_0(N))$. With notations as in the proof of Theorem~C, 
the residual representation $\bar{\rho}_f\simeq A[\mathfrak{P}]$ is ramified at some prime $q$, which we fix, and let $\cO_L$ be the ring of integers of $L$. 

Let $K=\bQ(\sqrt{-D})$ be an imaginary quadratic field satisfying:
\begin{itemize}
	\item[(a)] $q$ is ramified in $K$,
	\item[(b)] every prime factor $\ell\neq q$ of $N$ splits in $K$,
	\item[(c)] $p$ splits in $K$, 
	\item[(d)] $L(f^D,1)\neq 0$.
\end{itemize}

Let $\cO_f$ denote the ring of integers of $K_f$. By $(\ref{eq:ES})$ and (d), we have 
\[
{\rm ord}_{s=1}L(A/K,s)=[K_f:\bQ],
\] 
and hence by the work of Gross--Zagier and Kolyvagin it follows that:
\begin{itemize}
	\item[(i)] {} ${\rm rank}_{\cO_L}(A(K)\otimes_{\cO_f}\cO_L)=1$, and
	\item[(ii)] {} $\Sha(A/K)$ is finite.
\end{itemize} 
As in \cite[Thm.~2.3]{cas-mult}, by the results of \cite[\S{3.3}]{JSW}  conditions (i)-(ii) imply that 
$X^{{\rm rel},{\rm str}}_{K_\infty^{\rm ac}}(f)$ is $\Lambda(\Gamma^{\rm ac})$-torsion, and letting $f_{\rm ac}(T)\in\Lambda(\Gamma^{\rm ac})\simeq\cO_L[[T]]$ be a generator of its characteristic ideal we have
\begin{align*}
	\#\cO_L/f_{\rm ac}(0)&=\#\Sha(A/K)[\mathfrak{P}^\infty]\cdot
	\biggl(\frac{\#\cO_L/((\frac{1-a_p+p}{p})\log_{\omega_f}P)}{[A(K)\otimes_{\cO_f}\cO_L:\cO_L.P]}\biggr)^2\cdot\prod_{w}c_w(A/K)_{\mathfrak{P}},
\end{align*}
where $P\in A(K)$ is any point of infinite order, and $c_w^{}(A/K)_{\mathfrak{P}}$ is the $\mathfrak{P}$-part of the Tamagawa number of $A/K$ at $w$. Here the product is over primes $w$ dividing $N/q$, but since $\bar{\rho}_f$ is ramified at $q$, by \cite[Lem.~6.3]{zhang-Kolyvagin} there is no contribution from the prime above $q$. Taking for $P$ the Heegner point appearing in \cite[Prop.~5.1.7]{JSW} (which has infinite order by the Gross--Zagier formula \cite{YZZ}), the result follows just as in the aforementioned references, using Theorem~C to descend from $K$ to $\bQ$. 
\end{proof}

\bibliographystyle{amsalpha}
\bibliography{CCSS-refs}

\end{document}